\documentclass[11pt]{amsart}%
\usepackage{amsmath}
\usepackage{amssymb}
\usepackage{amsthm}
\usepackage[mathscr]{eucal}
\usepackage{mathrsfs}
\usepackage{psfrag}
\usepackage{fullpage}
\usepackage{color}
\usepackage{eucal}
\usepackage[dvips]{graphicx}
\usepackage{amsfonts}%
\usepackage{amsaddr}
\providecommand{\U}[1]{\protect\rule{.1in}{.1in}}
\newtheorem{proposition}{Proposition}[section]
\newtheorem{theorem}[proposition]{Theorem}
\newtheorem{corollary}[proposition]{Corollary}
\newtheorem{lemma}[proposition]{Lemma}

\newtheorem{definition}[proposition]{Definition}
\newtheorem{remark}[proposition]{Remark}

\newtheorem{condition}[proposition]{Condition}

\numberwithin{equation}{section}
\numberwithin{proposition}{section}

\begin{document}
\title{Large Deviations for Multiscale Diffusions via Weak Convergence Methods}
\author{Paul Dupuis and Konstantinos Spiliopoulos}
\address{Lefschetz Center for Dynamical Systems, Division of Applied Mathematics, Brown University, Providence, RI 02912}
\email{Paul\_Dupuis@brown.edu}
\email{kspiliop@dam.brown.edu}
\thanks{Research of P.D. supported in part by the National Science Foundation
(DMS-0706003, DMS-1008331), the Department of Energy (DE-SCOO02413), and the
Army Research Office (W911NF-09-1-0155).}
\thanks{Research of K.S. supported in part by the Department of Energy (DE-SCOO02413).}
\maketitle

\begin{abstract}
We study the large deviations principle for locally periodic stochastic
differential equations with small noise and fast oscillating coefficients.
There are three possible regimes depending on how fast the intensity of the
noise goes to zero relative to the homogenization parameter. We use weak
convergence methods which provide convenient representations for the action
functional for all three regimes. Along the way we study weak limits of
related controlled SDEs with fast oscillating coefficients and derive, in some
cases, a control that nearly achieves the large deviations lower bound at the
prelimit level. This control is useful for designing efficient importance
sampling schemes for multiscale diffusions driven by small noise.

\end{abstract}

\section{Introduction}

The purpose of this paper is to obtain large deviation properties of
stochastic differential equations with rapidly fluctuating coefficients in a
form that can be used for accelerated Monte Carlo. Such results are not
available in the literature. We use methods from weak convergence and
stochastic control. Consider the $d$-dimensional process $X^{\epsilon}%
\doteq\{X_{t}^{\epsilon},0\leq t\leq1\}$ satisfying the stochastic
differential equation (SDE)
\begin{equation}
dX_{t}^{\epsilon}=\left[  \frac{\epsilon}{\delta}b\left(  X_{t}^{\epsilon
},\frac{X_{t}^{\epsilon}}{\delta}\right)  +c\left(  X_{t}^{\epsilon}%
,\frac{X_{t}^{\epsilon}}{\delta}\right)  \right]  dt+\sqrt{\epsilon}%
\sigma\left(  X_{t}^{\epsilon},\frac{X_{t}^{\epsilon}}{\delta}\right)
dW_{t},\hspace{0.2cm}X_{0}^{\epsilon}=x_{0}, \label{Eq:LDPandA1}%
\end{equation}
where $\delta=\delta(\epsilon)\downarrow0$ as $\epsilon\downarrow0$ and
$W_{t}$ is a standard $d$-dimensional Wiener process. The functions
$b(x,y),c(x,y)$ and $\sigma(x,y)$ are assumed to be smooth according to
Condition \ref{A:Assumption1} and periodic with period 1 in every direction
with respect to the second variable.

If $\delta$ is of order $1$ while $\epsilon$ tends to zero, large deviations
theory tells how quickly (\ref{Eq:LDPandA1}) converges to the deterministic
ODE given by setting $\epsilon$ equal to zero. If $\epsilon$ is of order $1$
while $\delta$ tends to zero, homogenization occurs and one obtains an
equation with homogenized coefficients. If the two parameters go to zero
together then one expects different behaviors depending on how fast $\epsilon$
goes to zero relative to $\delta$.

Using the weak convergence approach of \cite{DupuisEllis}, we investigate the
large deviations principle (LDP) of $X^{\epsilon}$ under the following three
regimes:
\begin{equation}
\lim_{\epsilon\downarrow0}\frac{\epsilon}{\delta}=%
\begin{cases}
\infty & \text{Regime 1,}\\
\gamma\in(0,\infty) & \text{Regime 2,}\\
0 & \text{Regime 3.}%
\end{cases}
\label{Def:ThreePossibleRegimes}%
\end{equation}
The weak convergence approach results in a convenient representation formula
for the large deviations action functional (otherwise known as the rate
function) for all three regimes (Theorem \ref{T:MainTheorem2}). It is based on
the representation Theorem \ref{T:RepresentationTheorem}, which in this case
involves controlled SDE's with fast oscillating coefficients. Along the way,
we obtain a uniform proof of convergence of the underlying controlled SDE
(CSDE) in all three regimes (Theorem \ref{T:MainTheorem1}). In addition, in
some cases we construct a control that nearly achieves the large deviations
lower bound at the prelimit level. This control is useful, in particular, for
the design of efficient importance sampling schemes. The particular use of the
control will appear elsewhere.

A motivation for this work comes from chemical physics and biology, and in
particular from the dynamical behavior of proteins such as their folding and
binding kinetics. It was suggested long ago (e.g., \cite{LifsonJackson}) that
the potential surface of a protein might have a hierarchical structure with
potential minima within potential minima. The underlying energy landscapes of
certain biomolecules can be rugged (i.e., consist of many minima separated by
barriers of varying heights) due to the presence of multiple energy scales
associated with the building blocks of proteins. Roughness of the energy
landscapes that describe proteins has numerous effects on their folding and
binding as well as on their behavior at equilibrium. Often, these phenomena
are described mathematically by diffusion in a rough potential where a smooth
function is superimposed by a rough function (see Figure \ref{F:Figure1}). A
representative, but by no means complete, list of references is
\cite{Ansari,BryngelsonOnuchicWolynes, HyeonThirumalai, MondalGhosh,
SavenWangWolynes, Zwanzig}. The situation investigated in these papers is only
a special case of equation (\ref{Eq:LDPandA1}) with $\sigma(x,y)=\sqrt{2D}$,
$b(x,y)=-\frac{2D}{k_{\beta}T}\nabla Q(y)$ and $c(x,y)=-\frac{2D}{k_{\beta}%
T}\nabla V(y)$, where $k_{\beta}$ is the Boltzmann constant and $T$ is the
temperature. The questions of interest in these papers are related to the
effect of taking $\delta\downarrow0$ with $\epsilon$ small but fixed. This is
almost the same to requiring that $\delta$ goes to $0$ much faster than
$\epsilon$ does. Our goal is to study the related large deviations principle,
so we take $\epsilon\downarrow0$ as well. It will become clear that the
formula for the effective diffusivity (denoted by $q$ in Corollary
\ref{C:MainCorollary2}) that appears in the aforementioned chemistry and
biology literature is obtained under Regime $1$.

Singularly perturbed stochastic control problems and related large
deviations problems have been studied elsewhere (see for example
\cite{Baldi, BorkarGaitsgory, FS, GaitsgoryNguyen, Kushner,
Kushner1, Lipster, PardouxVeretennikov1, Veretennikov,
VeretennikovSPA2000} and the references therein). In particular, in
\cite{FS} the authors study the large deviation problem for periodic
coefficients, i.e., $b(x,y)=b(y),c(x,y)=c(y)$ and
$\sigma(x,y)=\sigma(y)$, using other methods. In \cite{FS}, the
authors provide an explicit formula for the action functional in
Regime $1$, whereas in Regimes $2$ and $3$ the action functional is
in terms of solutions to variational problems. In the present paper,
we derive the same explicit expression for the action functional in
Regime $1$. In addition, we also obtain the related control that
nearly achieves the LDP lower bound at the prelimit level. For
Regimes $2$ and $3$ we provide an alternative expression, from
\cite{FS}, for the action functional (Theorem \ref{T:MainTheorem2}).
It follows from these expressions that Regime $3$ can be seen as a
limiting case of Regime $2$ by simply setting $\gamma=0$, though we
are able to prove the large deviation lower bound in Regime 3 only
under additional conditions. For both regimes we derive explicit
expressions for the action functional in special cases of interest,
and in Regime 2 obtain a corresponding control that nearly achieves
the LDP lower bound. Note that the extension of the results of
\cite{FS} for Regime 2 to include the $x-$dependence is non-trivial,
since several smoothness properties of the local rate function need
to be proven (see Subsection \ref{S:BoundedOptimalControlRegime2}
for details). Apart from \cite{FS}, Regime $2$ has also been studied
in \cite{Kushner1, Veretennikov, VeretennikovSPA2000} under various
assumptions and dependencies of the coefficients of the system on
the slow and fast motion. In \cite{FS, Veretennikov,
VeretennikovSPA2000}, the local rate function is characterized as
the Legendre-Fenchel transform of the limit of the normalized
logarithm of an exponential moment or of the first eigenvalue of an
associated operator. In the present paper, we provide a direct
expression for the local rate function (Theorem
\ref{T:MainTheorem4Regime2}).

We note here that in the case of Regime $1$ one can weaken the periodicity
assumption, using the results of \cite{PardouxVeretennikov1} and the
methodology of the present paper, and prove an analogous result when the fast
variable takes values in $\mathbb{R}^{d}$. It also seems possible to combine
the methods of the present paper together with results in \cite{KaiseSheu,
BensoussanFrehse} to weaken the periodicity assumption for Regime $2$ as well;
see Remark \ref{R:TheWholeEuclideanSpace} for more details.

The paper is organized as follows. In Section \ref{S:Main}, we establish
notation, review some preliminary results and state the general large
deviations result (Theorem \ref{T:MainTheorem2}). Section \ref{S:Limit}
considers the weak limit of the associated controlled stochastic differential
equations. In Section \ref{S:LowerBoundLowerSemicontinuity} we prove the large
deviations upper bound for all three regimes and the compactness of the level
sets of the rate function. Section \ref{S:LaplacePrincipleRegime1} contains
the proof of the large deviations lower bound (or equivalently Laplace
principle upper bound) for Regime $1$, which completes the proof of the large
deviations principle for Regime $1$. This section also discusses an explicit
expression for a control that nearly achieves the large deviations lower bound
in the prelimit level $(\epsilon>0)$. In Section
\ref{S:LaplacePrincipleRegime2}, we prove the large deviations lower bound for
Regime $2$ and identify a control that nearly achieves this lower bound.
Section \ref{S:LaplacePrincipleRegime3} discusses the large deviations lower
bound principle for Regime $3$ and presents alternative expressions for the
rate function in dimension $1$.

\section{Preliminaries, statement of the main results.}

\label{S:Main}

We work with the canonical filtered probability space $(\Omega,\mathfrak{F}%
,\mathbb{P})$ equipped with a filtration $\mathfrak{F}_{t}$ that satisfies the
usual conditions, namely, $\mathfrak{F}_{t}$ is right continuous and
$\mathfrak{F}_{0}$ contains all $\mathbb{P}$-negligible sets.

In preparation for stating the main results, we recall the concept of a
Laplace principle. Throughout this paper only random variables that take
values in a Polish space are considered. By definition, a rate function on a
Polish space $\mathcal{S}$ maps $\mathcal{S}$ into $[0,\infty]$ and has
compact level sets.

\begin{definition}
\label{Def:LaplacePrinciple} Let $\{X^{\epsilon},\epsilon>0\}$ be a family of
random variables taking values in $\mathcal{S}$ and let $I$ be a rate function
on $\mathcal{S}$. We say that $\{X^{\epsilon},\epsilon>0\}$ satisfies the
Laplace principle with rate function $I$ if for every bounded and continuous
function $h:\mathcal{S}\rightarrow\mathbb{R}$
\begin{equation*}
\lim_{\epsilon\downarrow0}-\epsilon\ln\mathbb{E}\left[  \exp\left\{
-\frac{h(X^{\epsilon})}{\epsilon}\right\}  \right]  =\inf_{x\in\mathcal{S}%
}\left[  I(x)+h(x)\right]  . \label{Eq:LaplacePrinciple}%
\end{equation*}

\end{definition}

A Laplace principle is equivalent to the corresponding large deviations
principle with the same rate function (if the definition of a rate function
includes the requirement of compact level sets, see Theorems 2.2.1 and 2.2.3
in \cite{DupuisEllis}). Thus instead of proving a large deviations principle
for $\{X^{\epsilon}\}$ we prove a Laplace principle for $\{X^{\epsilon}\}$.

Regarding the SDE (\ref{Eq:LDPandA1}) we impose the following condition.

\begin{condition}
\label{A:Assumption1}

\begin{enumerate}
\item The functions $b(x,y),c(x,y),\sigma(x,y)$ are Lipschitz continuous and
bounded in both variables and periodic with period $1$ in the second variable
in each direction. In the case of Regime $1$ we additionally assume that they
are $C^{1}(\mathbb{R}^{d})$ in $y$ and $C^{2}(\mathbb{R}^{d})$ in $x$ with all
partial derivatives continuous and globally bounded in $x$ and $y $.

\item The diffusion matrix $\sigma\sigma^{T}$ is uniformly nondegenerate.

\end{enumerate}
\end{condition}

The regularity conditions imposed are stronger than necessary, but they are
assumed to simplify the exposition. See Remark \ref{R:RegularityCondition} for
some further details on this. For notational convenience we define the
operator $\cdot:\cdot$, where for two matrices $A=[a_{ij}],B=[b_{ij}]$
\[
A:B\doteq\sum_{i,j}a_{ij}b_{ij}.
\]
Under Regime $1$, we also impose the following condition.

\begin{condition}
\label{A:Assumption2} Let $\mu(dy|x)$ be the unique invariant measure
corresponding to the operator
\begin{equation*}
\mathcal{L}_{x}^{1}=b(x,y)\cdot\nabla_{y}+\frac{1}{2}\sigma(x,y)\sigma
(x,y)^{T}:\nabla_{y}\nabla_{y} \label{OperatorRegime1}%
\end{equation*}
equipped with periodic boundary conditions in $y$ ($x$ is being treated as a
parameter here). Under Regime 1, we assume the standard centering condition
(see \cite{BLP}) for the unbounded drift term $b$:
\[
\int_{\mathcal{Y}}b(x,y)\mu(dy|x)=0,
\]
where $\mathcal{Y}=\mathbb{T}^{d}$ denotes the $d$-dimensional torus.
\end{condition}

We note that under Conditions \ref{A:Assumption1} and \ref{A:Assumption2}, for
each $\ell\in\{1,\ldots,d\}$ there is a unique, twice differentiable function
$\chi_{\ell}(x,y)$ that is one periodic in every direction in $y$, that solves
the following cell problem (for a proof see \cite{BLP}, Theorem 3.3.4):
\begin{equation}
\mathcal{L}_{x}^{1}\chi_{\ell}(x,y)=-b_{\ell}(x,y),\quad\int_{\mathcal{Y}}%
\chi_{\ell}(x,y)\mu(dy|x)=0. \label{Eq:CellProblem}%
\end{equation}
We write $\chi=(\chi_{1},\ldots,\chi_{d})$.

\vspace{0.4cm}

Our tool for proving the Laplace principle will be the weak convergence
approach of \cite{DupuisEllis}. The following representation theorem is
essential for this approach. A proof of this theorem is given in
\cite{BoueDupuis}. The control process can depend on $\epsilon$ but this is
not always denoted explicitly. In the representation and elsewhere we take
$T=1$. Analogous results hold for arbitrary $T\in(0,\infty)$.

\begin{theorem}
\label{T:RepresentationTheorem} Assume Condition \ref{A:Assumption1}, and
given $\epsilon>0$ let $X^{\epsilon}$ be the unique strong solution to
(\ref{Eq:LDPandA1}). Then for any bounded Borel measurable function $h$
mapping $\mathcal{C}([0,1];\mathbb{R}^{d})$ into $\mathbb{R}$
\[
-\epsilon\ln\mathbb{E}_{x_{0}}\left[  \exp\left\{  -\frac{h(X^{\epsilon}%
)}{\epsilon}\right\}  \right]  =\inf_{u\in\mathcal{A}}\mathbb{E}_{x_{0}%
}\left[  \frac{1}{2}\int_{0}^{1}\left\Vert u_{t}\right\Vert ^{2}dt+h(\bar
{X}^{\epsilon})\right]  ,
\]
where $\mathcal{A}$ is the set of all $\mathfrak{F}_{t}-$progressively
measurable $d$-dimensional processes $u\doteq\{u_{t},0\leq t\leq1\}$
satisfying
\begin{equation*}
\mathbb{E}\int_{0}^{1}\left\Vert u_{t}\right\Vert ^{2}dt<\infty,
\label{A:AdmissibleControls}%
\end{equation*}
and $\bar{X}^{\epsilon}$ is the unique strong solution to
\begin{equation}
d\bar{X}_{t}^{\epsilon}=\left[  \frac{\epsilon}{\delta}b\left(  \bar{X}%
_{t}^{\epsilon},\frac{\bar{X}_{t}^{\epsilon}}{\delta}\right)  +c\left(
\bar{X}_{t}^{\epsilon},\frac{\bar{X}_{t}^{\epsilon}}{\delta}\right)  \right]
dt+\sigma\left(  \bar{X}_{t}^{\epsilon},\frac{\bar{X}_{t}^{\epsilon}}{\delta
}\right)  u_{t}dt+\sqrt{\epsilon}\sigma\left(  \bar{X}_{t}^{\epsilon}%
,\frac{\bar{X}_{s}^{\epsilon}}{\delta}\right)  dW_{t},\hspace{0.2cm}\bar
{X}_{0}^{\epsilon}=x_{0}. \label{Eq:LDPandA2}%
\end{equation}

\end{theorem}

Before stating the main results, we need additional notation and definitions.
Let $\mathcal{Z}=\mathbb{R}^{d}$ denote the space in which the control process
takes values.

\begin{definition}
\label{Def:ThreePossibleOperators} For the three possible Regimes $i=1,2,3$
defined in (\ref{Def:ThreePossibleRegimes}) and for $x\in\mathbb{R}^{d}%
,y\in\mathcal{Y}$ and $z\in\mathcal{Z}$, let
\begin{align}
\mathcal{L}_{x}^{1}  &  =b(x,y)\cdot\nabla_{y}+\frac{1}{2}\sigma
(x,y)\sigma(x,y)^{T}:\nabla_{y}\nabla_{y}\nonumber\\
\mathcal{L}_{z,x}^{2}  &  =\left[  \gamma b(x,y)+c(x,y)+\sigma(x,y)z\right]
\cdot\nabla_{y}+\gamma\frac{1}{2}\sigma(x,y)\sigma(x,y)^{T}:\nabla_{y}%
\nabla_{y}\nonumber\\
\mathcal{L}_{z,x}^{3}  &  =\left[  c(x,y)+\sigma(x,y)z\right]  \cdot\nabla
_{y}.\nonumber
\end{align}
For $i=1,2$ we let $\mathcal{D}(\mathcal{L}_{z,x}^{i})=\mathcal{C}%
^{2}(\mathcal{Y})$ and for $i=3$, $\mathcal{D}(\mathcal{L}_{z,x}%
^{3})=\mathcal{C}^{1}(\mathcal{Y})$.
\end{definition}

We also define for Regime $i$ a function $\lambda_{i}(x,y,z)$, $i=1,2,3$, as follows.

\begin{definition}
\label{Def:ThreePossibleFunctions} For the three possible Regimes $i=1,2,3$
defined in (\ref{Def:ThreePossibleRegimes}) and for $x\in\mathbb{R}^{d}%
,y\in\mathcal{Y}$ and $z\in\mathcal{Z}$, define $\lambda_{i}(x,y,z):\mathbb{R}%
^{d}\times\mathcal{Y}\times\mathcal{Z}\rightarrow\mathbb{R}^{d}$ by
\begin{align}
\lambda_{1}(x,y,z)  &  =\left(  I+\frac{\partial\chi}{\partial y}(x,y)\right)
\left(  c(x,y)+\sigma(x,y)z\right) \nonumber\\
\lambda_{2}(x,y,z)  &  =\gamma b(x,y)+c(x,y)+\sigma(x,y)z\nonumber\\
\lambda_{3}(x,y,z)  &  =c(x,y)+\sigma(x,y)z,\nonumber
\end{align}
where $\chi=(\chi_{1},\ldots,\chi_{d})$ is defined by (\ref{Eq:CellProblem})
and $I$ is the identity matrix.
\end{definition}

For a Polish space $\mathcal{S}$, let $\mathcal{P}(\mathcal{S})$ be the space
of probability measures on $\mathcal{S}$. Let $\Delta=\Delta(\epsilon
)\downarrow0$ as $\epsilon\downarrow0$. The role of $\Delta(\epsilon)$ is to
exploit a time-scale separation.
Let $A,B,\Gamma$ be Borel sets of $\mathcal{Z},\mathcal{Y},[0,1]$
respectively. Let $u^{\epsilon}\in\mathcal{A}$ and let $\bar{X}_{s}^{\epsilon
}$ solve (\ref{Eq:LDPandA2}) with $u^{\epsilon}$ in place of $u$. We associate
with $\bar{X}^{\epsilon}$ and $u^{\epsilon}$ a family of occupation measures
$\mathrm{P}^{\epsilon,\Delta}$ defined by
\begin{equation}
\mathrm{P}^{\epsilon,\Delta}(A\times B\times\Gamma)=\int_{\Gamma}\left[
\frac{1}{\Delta}\int_{t}^{t+\Delta}1_{A}(u_{s}^{\epsilon})1_{B}\left(
\frac{\bar{X}_{s}^{\epsilon}}{\delta}\mod 1\right)  ds\right]  dt,
\label{Def:OccupationMeasures2}%
\end{equation}
with the convention that if $s>1$ then $u_{s}^{\epsilon}=0$.

The first result, Theorem \ref{T:MainTheorem1}, deals with the limiting
behavior of the controlled process (\ref{Eq:LDPandA2}) under each of the three
regimes, and uses the notion of a viable pair.

\begin{definition}
\label{Def:ViablePair} A pair $(\psi,\mathrm{P})\in\mathcal{C}%
([0,1];\mathbb{R}^{d})\times\mathcal{P}(\mathcal{Z}\times\mathcal{Y}%
\times\lbrack0,1])$ will be called viable with respect to $(\lambda
,\mathcal{L})$, or simply viable if there is no confusion, if the following
are satisfied. The function $\psi_{t}$ is absolutely continuous, $\mathrm{P}$
is square integrable in the sense that $\int_{\mathcal{Z}\times\mathcal{Y}%
\times\lbrack0,1]}\left\Vert z\right\Vert ^{2}\mathrm{P}(dzdyds)<\infty$, and
the following hold for all $t\in\lbrack0,1]$:
\begin{equation}
\psi_{t}=x_{0}+\int_{\mathcal{Z}\times\mathcal{Y}\times\lbrack0,t]}%
\lambda(\psi_{s},y,z)\mathrm{P}(dzdyds),
\label{Eq:AccumulationPointsProcessViable}%
\end{equation}
for every $f\in\mathcal{D}(\mathcal{L})$%
\begin{equation}
\int_{0}^{t}\int_{\mathcal{Z}\times\mathcal{Y}}\mathcal{L}_{z,\psi_{s}%
}f(y)\mathrm{P}(dzdyds)=0, \label{Eq:AccumulationPointsMeasureViable}%
\end{equation}
and
\begin{equation}
\mathrm{P}(\mathcal{Z}\times\mathcal{Y}\times\lbrack0,t])=t.
\label{Eq:AccumulationPointsFullMeasureViable}%
\end{equation}
We write $(\psi,\mathrm{P})\in\mathcal{V}_{(\lambda,\mathcal{L})}$ or simply
$(\psi,\mathrm{P})\in\mathcal{V}$ if there is no confusion.
\end{definition}

Equation (\ref{Eq:AccumulationPointsFullMeasureViable}) implies that the last
marginal of $\mathrm{P}$ is Lebesgue measure, and hence $\mathrm{P}$ can be
decomposed in the form $\mathrm{P}(dzdydt)=\mathrm{P}_{t}(dzdy)dt$. Equations
(\ref{Eq:AccumulationPointsFullMeasureViable}) and
(\ref{Eq:AccumulationPointsMeasureViable}) then imply that, for a choice of
the kernel $\mathrm{P}_{t}(dzdy)$, $\mathrm{P}_{t}(\mathcal{Z}\times
\mathcal{Y})=1$ and
\[
\int_{\mathcal{Z}\times\mathcal{Y}}\mathcal{L}_{z,\psi_{t}}f(y)\mathrm{P}%
_{t}(dzdy)=0,
\]
and by (\ref{Eq:AccumulationPointsProcessViable}) for a.e. $t\in\lbrack0,1] $
\begin{equation*}
\dot{\psi}_{t}=\int_{\mathcal{Z}\times\mathcal{Y}}\lambda(\psi_{t}%
,y,z)\mathrm{P}_{t}(dzdy). \label{Eq:AccumulationPointsProcessViableD}%
\end{equation*}
Note that a viable pair depends on the initial condition $\psi_{0}=x_{0}$ as
well. Since this is will be deterministic and fixed throughout the paper, we
frequently omit writing this dependence explicitly.

\begin{theorem}
\label{T:MainTheorem1} Given $x_{0}\in\mathbb{R}^{d}$, consider any family
$\{u^{\epsilon},\epsilon>0\}$ of controls in $\mathcal{A}$ satisfying
\begin{equation*}
\sup_{\epsilon>0}\mathbb{E}\int_{0}^{1}\left\Vert u_{t}^{\epsilon}\right\Vert
^{2}dt<\infty\label{Eq:Ubound}%
\end{equation*}
and assume Condition \ref{A:Assumption1}. In addition, in Regime 1 assume
Condition \ref{A:Assumption2}. Then the family $\{(\bar{X}^{\epsilon
},\mathrm{P}^{\epsilon,\Delta}),\epsilon>0\}$ is tight. Hence given Regime
$i$, $i=1,2,3$, and given any subsequence of $\{(\bar{X}^{\epsilon}%
,\mathrm{P}^{\epsilon,\Delta}),\epsilon>0\}$, there exists a subsubsequence
that converges in distribution with limit $(\bar{X}^{i},\mathrm{P}^{i})$. With
probability $1$, the accumulation point $(\bar{X}^{i},\mathrm{P}^{i})$ is a
viable pair with respect to $(\lambda_{i},\mathcal{L}^{i})$ according to
Definition \ref{Def:ViablePair}, i.e., $(\bar{X}^{i},\mathrm{P}^{i}%
)\in\mathcal{V}_{(\lambda_{i},\mathcal{L}^{i})}$.
\end{theorem}

A proof is given in Section \ref{S:Limit}. The following theorem is the main
result of this paper. It asserts that a large deviation principle holds, and
gives a unifying expression for the rate function for all three regimes.

\begin{theorem}
\label{T:MainTheorem2} Let $\{X^{\epsilon},\epsilon>0\}$ be the
unique strong solution to (\ref{Eq:LDPandA1}). Assume Condition
\ref{A:Assumption1} and that we are considering Regime $i$, where
$i=1,2,3$. In Regime 1 assume Condition \ref{A:Assumption2} and in
Regime $3$ assume either that we are in dimension $d=1$, or that
$c(x,y)=c(y)$ and $\sigma(x,y)=\sigma(y)$ for the general
multidimensional case. Define
\begin{equation}
S^{i}(\phi)=\inf_{(\phi,\mathrm{P})\in\mathcal{V}_{(\lambda_{i},\mathcal{L}%
^{i})}}\left[  \frac{1}{2}\int_{\mathcal{Z}\times\mathcal{Y}\times\lbrack
0,1]}\left\Vert z\right\Vert ^{2}\mathrm{P}(dzdydt)\right]  ,
\label{Eq:GeneralRateFunction}%
\end{equation}
with the convention that the infimum over the empty set is $\infty$. Then for
every bounded and continuous function $h$ mapping $\mathcal{C}%
([0,1];\mathbb{R}^{d})$ into $\mathbb{R}$
\begin{equation*}
\lim_{\epsilon\downarrow0}-\epsilon\ln\mathbb{E}_{x_{0}}\left[  \exp\left\{
-\frac{h(X^{\epsilon})}{\epsilon}\right\}  \right]  =\inf_{\phi\in
\mathcal{C}([0,1];\mathbb{R}^{d})}\left[  S^{i}(\phi)+h(\phi)\right]  .
\label{Eq:LaplacePrincipleTheorem}%
\end{equation*}
Moreover, for each $s<\infty$, the set
\begin{equation*}
\Phi_{s}^{i}=\{\phi\in\mathcal{C}([0,1];\mathbb{R}^{d}):S^{i}(\phi)\leq s\}
\label{Def:LevelSets}%
\end{equation*}
is a compact subset of $\mathcal{C}([0,1];\mathbb{R}^{d})$. In other words,
$\{X^{\epsilon},\epsilon>0\}$ satisfies the Laplace principle with rate
function $S^{i}$.
\end{theorem}

The proof of this theorem is given in the subsequent sections. In Section
\ref{S:LaplacePrincipleRegime1} we prove that the formulation given in
(\ref{Eq:GeneralRateFunction}) for the rate function takes an explicit form in
Regime $1$ which agrees with the formula provided in \cite{FS}. We also
construct a nearly optimal control that achieves the LDP lower bound (or
equivalently the Laplace principle upper bound) at the prelimit level, see
Theorem \ref{T:MainTheorem3}. In Sections \ref{S:LaplacePrincipleRegime2} and
\ref{S:LaplacePrincipleRegime3}, similar constructions are provided for
Regimes $2$ and $3$, respectively.

\begin{remark}
In the case of Regime $3$ we prove the Laplace principle lower bound
for the general d-dimensional $(x,y)$-dependent case. However, for
reasons that will be explained in Section
\ref{S:LaplacePrincipleRegime3}, we can prove the Laplace principle
upper bound for the general $(x,y)-$dependent case in dimension
$d=1$ and  under the assumption that $c$ and $\sigma$ are
independent of $x$ for the general multidimensional case. We
conjecture that the full Laplace principle holds without this
restriction, and note also that the rate function for Regime $3$ is
a limiting case of that of Regime $2$ obtained by setting
$\gamma=0$.
\end{remark}

\vspace{0.2cm}

\begin{remark}
\label{R:RegularityCondition} The regularity assumptions imposed in Condition
\ref{A:Assumption1} can be relaxed. Due to Condition \ref{A:Assumption1}, the
solution to the cell problem (\ref{Eq:CellProblem}) is twice differentiable,
which allows us to apply It\^{o}'s formula. Consider the case $b=b(y),\sigma
=\sigma(y)$ and assume that they are Lipschitz continuous. Then, standard
elliptic regularity theory (e.g., \cite{GilbargTrudinger}) shows that the
solution $\chi$ to equation (\ref{Eq:CellProblem}) is in $H^{2}(\mathcal{Y}%
)=W^{2,2}(\mathcal{Y})$. By Sobolev's embedding lemma it is also in
$C^{1}(\mathcal{Y})$. Then, using a standard approximation argument, one can
still prove Theorems \ref{T:MainTheorem1} and \ref{T:MainTheorem2} for Regime
$1$.
\end{remark}

\vspace{0.2cm}

We conclude this section with a remark on possible extensions of Theorem
\ref{T:MainTheorem2} to the case $\mathcal{Y}=\mathbb{R}^{d}$.

\begin{remark}
\label{R:TheWholeEuclideanSpace} In the case of Regime 1 and under some
additional assumptions, one can extend the results to $\mathcal{Y}%
=\mathbb{R}^{d}$. In particular, one needs to impose structural assumptions on
the coefficients $b$ and $\sigma$ such that an invariant measure corresponding
to the operator $\mathcal{L}_{x}^{1}$ exists. Also, note that for
$\mathcal{Y}=\mathbb{R}^{d}$ there are no boundary conditions associated with
the cell problem (\ref{Eq:CellProblem}). One looks for solutions that grow at
most polynomially in $y$, as $\left\Vert y\right\Vert \rightarrow\infty$. For
more details and specific statements on homogenization for fast oscillating
diffusion processes on the whole space, see \cite{PardouxVeretennikov1}. Using
these results and techniques similar to the ones developed in the current
paper, one can prove results that are analogous to Theorem
\ref{T:MainTheorem1} and Theorem \ref{T:MainTheorem2} for Regime $1$ and
$\mathcal{Y}=\mathbb{R}^{d}$.

The situation is a bit more complicated for Regimes $2$ and $3$. One of the
main reasons is that the operators $\mathcal{L}_{z,x}^{2}$ and $\mathcal{L}%
_{z,x}^{3}$ involve the control variable as well. However, using results on
the structure of solutions to ergodic type Bellman equations in $\mathbb{R}%
^{d}$ analogous to \cite{KaiseSheu, BensoussanFrehse} and techniques similar
to the ones developed in the current paper, it is seems possible that one can
prove a result that is analogous to Theorem \ref{T:MainTheorem2} for Regime
$2$ and $\mathcal{Y}=\mathbb{R}^{d}$. Ergodic type Bellman equations arise
naturally in the study of the local rate function in Subsection
\ref{S:BoundedOptimalControlRegime2}. Assuming special structure on the
dynamics, the authors in \cite{Kushner1} and \cite{VeretennikovSPA2000} have
looked at similar problems corresponding to Regime $2$ when $\mathcal{Y}%
=\mathbb{R}^{d}$, using other methods. Among other assumptions, the author in
\cite{Kushner1} assumes that the fast variable enters the equations of motion
in an affine fashion, whereas the author in \cite{VeretennikovSPA2000} assumes
that the diffusion coefficient of the fast motion is independent of the slow
motion. However, the arguments used in \cite{Kushner1, VeretennikovSPA2000} do
not seem to directly extend to the full nonlinear case.
\end{remark}

\section{Limiting behavior of the controlled process}

\label{S:Limit} In this section we prove Theorem \ref{T:MainTheorem1}. In
particular, in Subsection \ref{SS:Tightness} we prove tightness of the pair
$(\bar{X}^{\epsilon},\mathrm{P}^{\epsilon,\Delta})$ and in Subsection
\ref{SS:ConvergenceAllRegime} we prove that any accumulation point $(\bar
{X}^{i},\mathrm{P}^{i})$ of $(\bar{X}^{\epsilon},\mathrm{P}^{\epsilon,\Delta
})$ is a viable pair according to Definition \ref{Def:ViablePair} for Regimes
$i=1,2,3$. Note that the approach is the same for all three regimes.
Therefore, we present the proof in detail for Regime $1$ and for Regimes $2$
and $3$ only outline the differences.

\subsection{Tightness}

\label{SS:Tightness}

In this section we prove that the pair $(\bar{X}^{\epsilon},\mathrm{P}%
^{\epsilon,\Delta})$ is tight. The proof is independent of the regime under consideration.

\begin{proposition}
\label{P:tightness}Consider any family $\{u^{\epsilon},\epsilon>0\}$ of
controls in $\mathcal{A}$ satisfying
\begin{equation}
\sup_{\epsilon>0}\mathbb{E}\int_{0}^{1}\left\Vert u_{t}^{\epsilon}\right\Vert
^{2}dt<\infty\label{A:UniformlyAdmissibleControls}%
\end{equation}
and assume Condition \ref{A:Assumption1}. In addition, in Regime 1 assume
Condition \ref{A:Assumption2}. Then the following hold.

\begin{enumerate}
\item {The family $\{(\bar{X}^{\epsilon},\mathrm{P}^{\epsilon,\Delta
}),\epsilon>0\}$ is tight.}

\item {The family $\{\mathrm{P}^{\epsilon,\Delta},\epsilon>0\}$ is uniformly
integrable in the sense that
\[
\lim_{M\rightarrow\infty}\sup_{\epsilon>0}\mathbb{E}_{x_{0}}\left[
\int_{\{z\in\mathcal{Z}:\parallel z\parallel>M\}\times\mathcal{Y}\times
\lbrack0,1]}\left\Vert z\right\Vert \mathrm{P}^{\epsilon,\Delta}%
(dzdydt)\right]  =0.
\]
}
\end{enumerate}
\end{proposition}

\begin{proof}
\noindent(i). Tightness of the family $\{\bar{X}^{\epsilon}\}$ is standard if
we take into account the assumptions on the coefficients and the fact that the
sequence of controls $\{u^{\epsilon},\epsilon>0\}$ in $\mathcal{A}$ satisfy
(\ref{A:UniformlyAdmissibleControls}). Some care is needed only for Regime
$1$, because of the presence of the unbounded drift term. Recall that
$\chi=(\chi_{1},\ldots,\chi_{d})$ is one periodic in every direction in $y$
and satisfies
\begin{equation}
\mathcal{L}_{x}^{1}\chi_{\ell}(x,y)=-b_{\ell}(x,y),\quad\int_{\mathcal{Y}}%
\chi_{\ell}(x,y)\mu(dy|x)=0,\hspace{0.1cm}\ell=1,...,d.\nonumber
\end{equation}
Applying It\^{o}'s formula to $\chi(x,x/\delta)=(\chi_{1}(x,x/\delta
),\ldots,\chi_{d}(x,x/\delta))$ with $x=\bar{X}_{t}^{\epsilon}$, we get
\begin{align}
\bar{X}_{t}^{\epsilon}  &  =x_{0}+\int_{0}^{t}\left(  I+\frac{\partial\chi
}{\partial y}\right)  \left(  \bar{X}_{s}^{\epsilon},\frac{\bar{X}%
_{s}^{\epsilon}}{\delta}\right)  \left[  c\left(  \bar{X}_{s}^{\epsilon}%
,\frac{\bar{X}_{s}^{\epsilon}}{\delta}\right)  +\sigma\left(  \bar{X}%
_{s}^{\epsilon},\frac{\bar{X}_{s}^{\epsilon}}{\delta}\right)  u_{s}^{\epsilon
}\right]  ds\nonumber\\
&  +\int_{0}^{t}\left[  \epsilon\frac{\partial\chi}{\partial x}b+\delta
\frac{\partial\chi}{\partial x}\left[  c+\sigma u_{s}^{\epsilon}\right]
+\epsilon\delta\frac{1}{2}\sigma\sigma^{T}:\frac{\partial^{2}\chi}{\partial
x^{2}}+\epsilon\frac{1}{2}\sigma\sigma^{T}:\frac{\partial^{2}\chi}{\partial
x\partial y}\right]  \left(  \bar{X}_{s}^{\epsilon},\frac{\bar{X}%
_{s}^{\epsilon}}{\delta}\right)  ds\label{Eq:LDPRegime1}\\
&  +\sqrt{\epsilon}\int_{0}^{t}\left[  \left(  I+\frac{\partial\chi}{\partial
y}\right)  \sigma+\delta\frac{\partial\chi}{\partial x}\sigma\right]  \left(
\bar{X}_{s}^{\epsilon},\frac{\bar{X}_{s}^{\epsilon}}{\delta}\right)
dW_{s}-\delta\left[  \chi\left(  \bar{X}_{t}^{\epsilon},\frac{\bar{X}%
_{t}^{\epsilon}}{\delta}\right)  -\chi\left(  x_{0},\frac{x_{0}}{\delta
}\right)  \right]  .\nonumber
\end{align}
From this representation, the boundedness of the coefficients and the second
derivatives of $\chi$ and assumption (\ref{A:UniformlyAdmissibleControls}), it
follows that for every $\eta>0$
\[
\lim_{\rho\downarrow0}\limsup_{\epsilon\downarrow0}\mathbb{P}_{x_{0}}\left[
\sup_{|t_{1}-t_{2}|<\rho,0\leq t_{1}<t_{2}\leq1}|\bar{X}_{t_{1}}^{\epsilon
}-\bar{X}_{t_{2}}^{\epsilon}|\geq\eta\right]  =0.
\]
This implies the tightness of $\{\bar{X}^{\epsilon}\}$.

It remains to prove tightness of the occupation measures $\{\mathrm{P}%
^{\epsilon,\Delta},\epsilon>0\}$. We claim that the function
\[
g(r)=\int_{\mathcal{Z}\times\mathcal{Y}\times\lbrack0,1]}\left\Vert
z\right\Vert ^{2}r(dzdydt),\hspace{0.2cm}r\in\mathcal{P}(\mathcal{Z}%
\times\mathcal{Y}\times\lbrack0,1])
\]
is a tightness function, i.e., it is bounded from below and its level sets
$R_{k}=\{r\in\mathcal{P}(\mathbb{R}^{d}\times\mathcal{Y}\times\lbrack
0,1]):g(r)\leq k\}$ are relatively compact for each $k<\infty$. To prove the
relative compactness, observe that Chebyshev's inequality implies
\[
\sup_{r\in R_{k}}r\left(  \{(z,y)\in\mathcal{Z}\times\mathcal{Y}:\left\Vert
z\right\Vert >M\}\times\lbrack0,1]\right)  \leq\sup_{r\in R_{k}}\frac
{g(r)}{M^{2}}\leq\frac{k}{M^{2}}.
\]
Hence, $R_{k}$ is tight and thus relatively compact as a subset of
$\mathcal{P}$.

Since $g$ is a tightness function, by Theorem A.3.17 of \cite{DupuisEllis}
tightness of $\{\mathrm{P}^{\epsilon,\Delta},\epsilon>0\}$ will follow if we
prove that
\[
\sup_{\epsilon\in(0,1]}\mathbb{E}_{x_{0}}\left[  g(\mathrm{P}^{\epsilon
,\Delta})\right]  <\infty.
\]
However, by (\ref{A:UniformlyAdmissibleControls})
\begin{align}
\sup_{\epsilon\in(0,1]}\mathbb{E}_{x_{0}}\left[  g(\mathrm{P}^{\epsilon
,\Delta})\right]   &  =\sup_{\epsilon\in(0,1]}\mathbb{E}_{x_{0}}\left[
\int_{\mathcal{Z}\times\mathcal{Y}\times\lbrack0,1]}\left\Vert z\right\Vert
^{2}\mathrm{P}^{\epsilon,\Delta}(dzdydt)\right] \nonumber\\
&  =\sup_{\epsilon\in(0,1]}\mathbb{E}_{x_{0}}\int_{0}^{1}\frac{1}{\Delta}%
\int_{t}^{t+\Delta}\left\Vert u_{s}^{\epsilon}\right\Vert ^{2}dsdt\nonumber\\
&  <\infty.\nonumber
\end{align}
(ii). This follows from the last display and
\[
\mathbb{E}_{x_{0}}\left[  \int_{\{z\in\mathcal{Z}:\parallel z\parallel
>M\}\times\mathcal{Y}\times\lbrack0,1]}\left\Vert z\right\Vert \mathrm{P}%
^{\epsilon,\Delta}(dzdydt)\right]  \leq\frac{1}{M}\mathbb{E}_{x_{0}}\left[
\int_{\mathcal{Z}\times\mathcal{Y}\times\lbrack0,1]}\left\Vert z\right\Vert
^{2}\mathrm{P}^{\epsilon,\Delta}(dzdydt)\right]  .
\]
This concludes the proof of the proposition.
\end{proof}

\subsection{Weak convergence analysis}

\label{SS:ConvergenceAllRegime} Before beginning the weak convergence analysis
we make an observation that is useful in the proofs for all three regimes.
Let
\begin{equation}
g(\epsilon)=%
\begin{cases}
\frac{\delta^{2}}{\epsilon} & \text{Regime 1,}\\
\epsilon & \text{Regime 2,}\\
\delta & \text{Regime 3,}%
\end{cases}
\label{Def:ThreePossibleRegimesFunction_g}%
\end{equation}
where we recall that $\delta=\delta(\epsilon)\downarrow0$ as $\epsilon
\downarrow0$. Then the particular relation between $\delta$ and $\epsilon$ in
each regime as given by (\ref{Def:ThreePossibleRegimes}) implies that
$g(\epsilon)\downarrow0$ as $\epsilon\downarrow0$. The process $\bar{Y}%
_{t}^{\epsilon}=\bar{X}_{t}^{\epsilon}/\delta$ satisfies
\begin{equation}
\bar{Y}_{t}^{\epsilon}=\frac{x_{0}}{\delta}+\int_{0}^{t}\left[  \frac
{\epsilon}{\delta^{2}}b\left(  \bar{X}_{s}^{\epsilon},\bar{Y}_{s}^{\epsilon
}\right)  +\frac{1}{\delta}c\left(  \bar{X}_{s}^{\epsilon},\bar{Y}%
_{s}^{\epsilon}\right)  +\frac{1}{\delta}\sigma\left(  \bar{X}_{s}^{\epsilon
},\bar{Y}_{s}^{\epsilon}\right)  u_{s}^{\epsilon}\right]  ds+\frac
{\sqrt{\epsilon}}{\delta}\int_{0}^{t}\sigma\left(  \bar{X}_{s}^{\epsilon}%
,\bar{Y}_{s}^{\epsilon}\right)  dW_{s}. \label{Eq:StretchedOutProcess}%
\end{equation}
Recall the operators $\mathcal{L}_{z,x}^{i}$ for $i=1,2,3$ as given in
Definition \ref{Def:ThreePossibleOperators}. Suppose that instead of
(\ref{Eq:StretchedOutProcess}) we consider the analogous equation with the
slow motion and control \textquotedblleft frozen,\textquotedblright\ i.e.,
with $\bar{X}_{s}^{\epsilon}$ replaced by $x$ and $u_{s}^{\epsilon}$ replaced
by $z$, and define $\mathcal{A}_{z,x}^{\epsilon}$ by
\begin{equation}
\mathcal{A}_{z,x}^{\epsilon}f(y)=\left[  \frac{\epsilon}{\delta^{2}%
}b(x,y)+\frac{1}{\delta}\left[  c(x,y)+\sigma(x,y)z\right]  \right]
\cdot\nabla_{y}f(y)+\frac{\epsilon}{\delta^{2}}\frac{1}{2}\sigma\sigma
^{T}(x,y):\nabla_{y}\nabla_{y}f(y) \label{Def_of_A_gen}%
\end{equation}
for suitable functions $f$. Then it is easy to check that $g(\epsilon
)\mathcal{A}_{z,x}^{\epsilon}$ converges to $\mathcal{L}_{z,x}^{i}$ under
Regime $i=1,3$ and to $\gamma\mathcal{L}_{z,x}^{2}$ under Regime $i=2$, as
$\epsilon\downarrow0$.

\subsubsection{Limiting behavior of the CSDE in Regime 1.}

\label{SS:ConvergenceRegime1} In this section we prove Theorem
\ref{T:MainTheorem1} for $i=1$. For notational convenience we drop the
subscript or superscript $1$ from $\lambda_{1},\bar{X}^{1}$ and $\mathrm{P}%
^{1}$.

\begin{lemma}
\label{L:LemmaForMainTheorem1} Let $T>0$ and $\tau>0$ be positive numbers such
that $T+\tau\leq1$. Consider a continuous function $g:\mathbb{R}^{d}%
\times\mathcal{Y}\times\mathcal{Z}\rightarrow\mathbb{R}$ that is bounded in
the first and the second argument and affine in the third argument. Assume
that $(\bar{X}^{\epsilon},\mathrm{P}^{\epsilon,\Delta})\rightarrow(\bar
{X},\mathrm{P})\hspace{0.1cm}$in distribution for some subsequence of
$\epsilon\downarrow0$, and that Conditions \ref{A:Assumption1} and
\ref{A:Assumption2} and (\ref{A:UniformlyAdmissibleControls}) hold. Then the
following limits are valid in distribution along this subsequence:%
\begin{equation}
\int_{\mathcal{Z}\times\mathcal{Y}\times\lbrack T,T+\tau]}g(\bar{X}%
_{t}^{\epsilon},y,z)\mathrm{P}^{\epsilon,\Delta}(dzdydt)\rightarrow
\int_{\mathcal{Z}\times\mathcal{Y}\times\lbrack T,T+\tau]}g(\bar{X}%
_{t},y,z)\mathrm{P}(dzdydt) \label{Eq:MartingaleProblemRegime1General1}%
\end{equation}
and
\begin{equation}
\int_{T}^{T+\tau}g\left(  \bar{X}_{t}^{\epsilon},\frac{\bar{X}_{t}^{\epsilon}%
}{\delta},u_{t}^{\epsilon}\right)  dt-\int_{\mathcal{Z}\times\mathcal{Y}%
\times\lbrack T,T+\tau]}g(\bar{X}_{t}^{\epsilon},y,z)\mathrm{P}^{\epsilon
,\Delta}(dzdydt)\rightarrow0. \label{Eq:MartingaleProblemRegime1General2}%
\end{equation}

\end{lemma}

\begin{proof}
First note that (\ref{Eq:MartingaleProblemRegime1General1}) holds due to the
weak convergence, the fact that the last marginal of $\mathrm{P}%
^{\epsilon,\Delta}(dzdydt)$ is always Lebesgue measure and part (ii) of
Proposition \ref{P:tightness} (see \cite{BoueDupuisEllis}, page 137 for more details).

Next we show that (\ref{Eq:MartingaleProblemRegime1General2}) holds. This
follows from the following three observations.

\begin{enumerate}
\item {Change of the order of integration implies that if $\tilde
{h}(s):[0,\infty)\rightarrow\mathbb{R}$ is integrable on each bounded interval
then
\begin{equation}
\left\vert \int_{0}^{T}\frac{1}{\Delta}\int_{t}^{t+\Delta}\tilde
{h}(s)dsdt-\int_{0}^{T}\tilde{h}(s)ds\right\vert \leq\int_{0}^{\Delta}%
|\tilde{h}(s)|ds+\int_{T}^{T+\Delta}|\tilde{h}(s)|ds. \label{Eq:Remark1}%
\end{equation}
}

\item {The definition of the occupation measure $\mathrm{P}^{\epsilon,\Delta}$
gives }%
\begin{equation}
{\int_{\mathcal{Z}\times\mathcal{Y}\times\lbrack T,T+\tau]}g(\bar{X}%
_{t}^{\epsilon},y,z)\mathrm{P}^{\epsilon,\Delta}(dzdydt)=\int_{T}^{T+\tau
}\frac{1}{\Delta}\int_{t}^{t+\Delta}g\left(  \bar{X}_{t}^{\epsilon},\frac
{\bar{X}_{s}^{\epsilon}}{\delta},u_{s}^{\epsilon}\right)  ds.}
\label{Eq:Remark2}%
\end{equation}

\item {The tightness of $\{\bar{X}^{\epsilon}\}$ implies that for every
$\eta>0$}%
\[
\sup_{0\leq t\leq T}\mathbb{P}_{x_{0}}\left[  \sup_{0\leq t\leq t+\Delta\leq
T}|\bar{X}_{t+\Delta}^{\epsilon}-\bar{X}_{t}^{\epsilon}|>\eta\right]
\rightarrow0\text{ as }\epsilon\downarrow0.
\]

\end{enumerate}

The last display, together with the continuity of $g$ in the first variable,
the fact that the second variable takes values in a compact space, and part
(ii) of Proposition \ref{P:tightness}, imply that
\[
{\int_{T}^{T+\tau}\frac{1}{\Delta}\int_{t}^{t+\Delta}g\left(  \bar{X}%
_{s}^{\epsilon},\frac{\bar{X}_{s}^{\epsilon}}{\delta},u_{s}^{\epsilon}\right)
dsdt-\int_{T}^{T+\tau}\frac{1}{\Delta}\int_{t}^{t+\Delta}g\left(  \bar{X}%
_{t}^{\epsilon},\frac{\bar{X}_{s}^{\epsilon}}{\delta},u_{s}^{\epsilon}\right)
dsdt\rightarrow0\ }\text{in probability.}%
\]
Then (\ref{Eq:Remark2}), the last display, (\ref{Eq:Remark1}) and
(\ref{A:UniformlyAdmissibleControls}) show that
(\ref{Eq:MartingaleProblemRegime1General2}) holds.
\end{proof}

\noindent\textit{Proof of Theorem \ref{T:MainTheorem1} for $i=1$.} The
tightness proven in Proposition \ref{P:tightness} implies that for any
subsequence of $\epsilon>0$ there exists a convergent subsubsequence and
$(\bar{X},\mathrm{P})$ such that
\[
(\bar{X}^{\epsilon},\mathrm{P}^{\epsilon,\Delta})\rightarrow(\bar
{X},\mathrm{P})\hspace{0.1cm}\text{in distribution.}%
\]
We invoke the Skorokhod representation theorem (Theorem 1.8 in
\cite{EithierKurtz}) which allows us to assume that the aforementioned
convergence holds with probability $1$. The Skorokhod representation theorem
involves the introduction of another probability space, but this distinction
is ignored in the notation. Note that by Fatou's Lemma
\begin{equation}
\mathbb{E}_{x_{0}}\int_{\mathcal{Z}\times\mathcal{Y}\times\lbrack
0,1]}\left\Vert z\right\Vert ^{2}\mathrm{P}(dzdydt)<\infty
\label{A:Assumption2_1}%
\end{equation}
and so $\int_{\mathcal{Z}\times\mathcal{Y}\times\lbrack0,1]}\left\Vert
z\right\Vert ^{2}\mathrm{P}(dzdydt)<\infty$ \ w.p.1. Thus it remains to show
that $(\bar{X},\mathrm{P})$ satisfy (\ref{Eq:AccumulationPointsProcessViable}%
), (\ref{Eq:AccumulationPointsMeasureViable}) and
(\ref{Eq:AccumulationPointsFullMeasureViable}).

Our tool for proving (\ref{Eq:AccumulationPointsProcessViable}) will be the
characterization of solutions to SDE's via the martingale problem
\cite{EithierKurtz}. Let $f,\phi_{j}$ be smooth, real valued functions with
compact support. For a measure $r\in\mathcal{P}(\mathcal{Z}\times
\mathcal{Y}\times\lbrack0,1])$ and $t\in\lbrack0,1]$, define
\[
(r,\phi_{j})_{t}=\int_{\mathcal{Z}\times\mathcal{Y}\times\lbrack0,t]}\phi
_{j}(z,y,s)r(dzdyds).
\]
Let $T,t_{i},\tau\geq0,i\leq q$ be given such that $t_{i}\leq T\leq T+\tau
\leq1$ and let $\zeta$ be a real valued, bounded and continuous function with
compact support on $(\mathbb{R}^{d})^{q}\times\mathbb{R}^{pq}$. We recall
that
\[
\lambda(x,y,z)=\left(  I+\frac{\partial\chi}{\partial y}(x,y)\right)
(c(x,y)+\sigma(x,y)z).
\]

In order to prove (\ref{Eq:AccumulationPointsProcessViable}), it is sufficient
to prove for any fixed such collection $p,q,T,t_{i},\tau,\phi_{j},\zeta,f$
that, as $\epsilon\downarrow0$,
\begin{equation}
\mathbb{E}_{x_{0}}\left[  \zeta(\bar{X}_{t_{i}}^{\epsilon},(\mathrm{P}%
^{\epsilon,\Delta},\phi_{j})_{t_{i}},i\leq q,j\leq p)\left[  f(\bar{X}%
_{T+\tau}^{\epsilon})-f(\bar{X}_{T}^{\epsilon})-\int_{T}^{T+\tau}%
\bar{\mathcal{A}}_{t}^{\epsilon,\Delta}f(\bar{X}_{t}^{\epsilon})dt\right]
\right]  \rightarrow0 \label{Eq:MartingaleProblemRegime1_1}%
\end{equation}
and
\begin{equation}
\int_{T}^{T+\tau}\bar{\mathcal{A}}_{s}^{\epsilon,\Delta}f(\bar{X}%
_{s}^{\epsilon})ds-\int_{\mathcal{Z}\times\mathcal{Y}\times\lbrack T,T+\tau
]}\lambda(\bar{X}_{s},y,z)\nabla f(\bar{X}_{s})\mathrm{P}(dzdyds)\rightarrow0
\label{Eq:MartingaleProblemRegime1_2a}%
\end{equation}
in probability. Here $\bar{\mathcal{A}}_{s}^{\epsilon,\Delta}$ is defined by
\begin{equation}
\bar{\mathcal{A}}_{t}^{\epsilon,\Delta}f(x)=\int_{\mathcal{Z}\times
\mathcal{Y}}\lambda(x,y,z)\nabla f(x)\mathrm{P}_{t}^{\epsilon,\Delta}(dzdy)
\label{Eq:PrelimitOperator}%
\end{equation}
and
\[
\mathrm{P}_{t}^{\epsilon,\Delta}(dzdy)=\frac{1}{\Delta}\int_{t}^{t+\Delta
}1_{dz}(u_{s}^{\epsilon})1_{dy}\left(  \frac{\bar{X}_{s}^{\epsilon}}{\delta
}\mod 1\right)  ds.
\]
Since they show that $(\bar{X},\mathrm{P})$ solves the appropriate martingale
problem, relations (\ref{Eq:MartingaleProblemRegime1_1}) and
(\ref{Eq:MartingaleProblemRegime1_2a}) imply
(\ref{Eq:AccumulationPointsProcessViable}). So, let us prove now that
(\ref{Eq:MartingaleProblemRegime1_1}) and
(\ref{Eq:MartingaleProblemRegime1_2a}) hold.

First, for every real valued, continuous function $\phi$ with compact support
and $t\in\lbrack0,1]$
\[
(\mathrm{P}^{\epsilon,\Delta},\phi)_{t}\rightarrow(\mathrm{P},\phi)_{t}\text{
w.p.}1.
\]
This follows from the topology used and the fact that the last marginal of
$\mathrm{P}$ is Lebesgue measure w.p.1. Second, we recall the solution
$\chi(x,y)$ to the cell problem (\ref{Eq:CellProblem}) and consider the
function $\psi_{\ell}(x,y)=\chi_{\ell}(x,y)f_{x_{\ell}}(x)$ for $\ell
=1,\ldots,d$. Then $\psi_{\ell}(x,y)$ is one periodic in every direction in
$y$ and satisfies
\begin{equation}
\mathcal{L}_{x}^{1}\psi_{\ell}(x,y)=-b_{\ell}(x,y)f_{x_{\ell}}(x),\quad
\int_{\mathcal{Y}}\psi_{\ell}(x,y)\mu(dy|x)=0. \label{Eq:CellProblem2}%
\end{equation}
Let $\psi=\{\psi_{1},\ldots,\psi_{d}\}$. We apply It\^{o}'s formula to
$\psi(\bar{X}_{s}^{\epsilon},\bar{X}_{s}^{\epsilon}/\delta)$. Relation
(\ref{Eq:CellProblem2}) and the boundedness of $\chi(x,y)$ and its derivatives
(see (\ref{Eq:LDPRegime1})) imply that in order to show
(\ref{Eq:MartingaleProblemRegime1_1}), it is sufficient to show that
\begin{equation}
\int_{T}^{T+\tau}\left[  \bar{\mathcal{A}}_{s}^{\epsilon,\Delta}f(\bar{X}%
_{s}^{\epsilon})-\lambda\left(  \bar{X}_{s}^{\epsilon},\frac{\bar{X}%
_{s}^{\epsilon}}{\delta},u_{s}^{\epsilon}\right)  \nabla f(\bar{X}%
_{s}^{\epsilon})\right]  ds\rightarrow0,\text{ as }\epsilon\downarrow0.
\label{Eq:MartingaleProblemRegime1_2}%
\end{equation}
in probability (a number of other terms converge to zero and we do not write
them explicitly for notational convenience). However, we can apply Lemma
\ref{L:LemmaForMainTheorem1} to
\[
g(x,y,z)=\lambda(x,y,z)\cdot\nabla f(x)=\left(  \left(  I+\frac{\partial\chi
}{\partial y}(x,y)\right)  c(x,y)+\left(  I+\frac{\partial\chi}{\partial
y}(x,y)\right)  \sigma(x,y)z\right)  \cdot\nabla f(x),
\]
in which case (\ref{Eq:MartingaleProblemRegime1_2a}) follows from
(\ref{Eq:MartingaleProblemRegime1General1}), and also
(\ref{Eq:MartingaleProblemRegime1_2}) (and hence
(\ref{Eq:MartingaleProblemRegime1_1})) follows from
(\ref{Eq:MartingaleProblemRegime1General2}). The completes the proof of
(\ref{Eq:AccumulationPointsProcessViable}).

Next we prove that (\ref{Eq:AccumulationPointsMeasureViable}) holds. For this
purpose define $\bar{Y}^{\epsilon}=\bar{X}^{\epsilon}/\delta$. Let $f_{\ell
}:\mathcal{Y}\mapsto\mathbb{R}$, $\ell\in\mathbb{N}$ be smooth and dense in
$\mathcal{C}(\mathcal{Y})$. Observe that the quantity
\begin{equation}
M_{t}^{\epsilon}=f_{\ell}(\bar{Y}_{t}^{\epsilon})-f_{\ell}(x_{0}/\delta
)-\int_{0}^{t}\mathcal{A}_{u_{s}^{\epsilon},\bar{X}_{s}^{\epsilon}}^{\epsilon
}f_{\ell}(\bar{Y}_{s}^{\epsilon})ds=\frac{\sqrt{\epsilon}}{\delta}\int_{0}%
^{t}\nabla_{y}f_{\ell}(\bar{Y}_{s}^{\epsilon})\cdot\sigma\left(  \bar{X}%
_{s}^{\epsilon},\bar{Y}_{s}^{\epsilon}\right)  dW_{s}, \label{Def:Martingale}%
\end{equation}
where $\mathcal{A}_{z,x}^{\epsilon}$ is defined in (\ref{Def_of_A_gen}), is an
$\mathfrak{F}_{t}-$martingale. Moreover, for any $T>0$, we have from
(\ref{Eq:Remark1}) that
\begin{align*}
\int_{0}^{T}\mathcal{A}_{u_{t}^{\epsilon},\bar{X}_{t}^{\epsilon}}^{\epsilon
}f_{\ell}(\bar{Y}_{t}^{\epsilon})dt+e_{T}^{\epsilon}  &  =\int_{0}^{T}\frac
{1}{\Delta}\left[  \int_{t}^{t+\Delta}\mathcal{A}_{u_{s}^{\epsilon},\bar
{X}_{s}^{\epsilon}}^{\epsilon}f_{\ell}(\bar{Y}_{s}^{\epsilon})ds\right]  dt\\
&  =\frac{\epsilon}{\delta^{2}}\int_{0}^{T}\frac{1}{\Delta}\left[  \int
_{t}^{t+\Delta}\mathcal{L}_{\bar{X}_{s}^{\epsilon}}^{1}f_{\ell}(\bar{Y}%
_{s}^{\epsilon})ds\right]  dt\\
&  +\frac{1}{\delta}\int_{0}^{T}\frac{1}{\Delta}\left[  \int_{t}^{t+\Delta
}\left[  c(\bar{X}_{s}^{\epsilon},\bar{Y}_{s}^{\epsilon})+\sigma(\bar{X}%
_{s}^{\epsilon},\bar{Y}_{s}^{\epsilon})u_{s}^{\epsilon}\right]  \cdot
\nabla_{y}f_{\ell}(\bar{Y}_{s}^{\epsilon})ds\right]  dt,
\end{align*}
where
\[
\left\vert e_{T}^{\epsilon}\right\vert \leq\int_{0}^{\Delta}|\mathcal{A}%
_{u_{s}^{\epsilon},\bar{X}_{s}^{\epsilon}}^{\epsilon}f_{\ell}(\bar{Y}%
_{s}^{\epsilon})|ds+\int_{T}^{T+\Delta}|\mathcal{A}_{u_{s}^{\epsilon},\bar
{X}_{s}^{\epsilon}}^{\epsilon}f_{\ell}(\bar{Y}_{s}^{\epsilon})|ds.
\]
Recall now the definition $g(\epsilon)=\delta^{2}/\epsilon\rightarrow0$ from
(\ref{Def:ThreePossibleRegimesFunction_g}) and define the operator
\[
\mathcal{G}_{x,y,z}f_{\ell}(y)=\left[  c(x,y)+\sigma(x,y)z\right]  \cdot
\nabla_{y}f_{\ell}(y).
\]

Let $D\subset\lbrack0,1]$ be countable and dense, and consider any $T\in D$
and $\ell\in\mathbb{N}$. By (\ref{Def:Martingale})%
\begin{align}
\lefteqn{g(\epsilon)M_{T}^{\epsilon}-g(\epsilon)\left[  f_{\ell}(\bar{Y}%
_{T}^{\epsilon})-f_{\ell}(\bar{Y}_{0}^{\epsilon})\right]  +g(\epsilon
)e_{T}^{\epsilon}}\nonumber\\
&  =\frac{1}{\delta}\int_{0}^{T}\frac{g(\epsilon)}{\Delta}\left[  \int
_{t}^{t+\Delta}\mathcal{G}_{\bar{X}_{s}^{\epsilon},\bar{Y}_{s}^{\epsilon
},u_{s}^{\epsilon}}f_{\ell}(\bar{Y}_{s}^{\epsilon})ds\right]  dt+\frac
{\epsilon}{\delta^{2}}\int_{0}^{T}\frac{g(\epsilon)}{\Delta}\left[  \int
_{t}^{t+\Delta}\mathcal{L}_{\bar{X}_{s}^{\epsilon}}^{1}f_{\ell}(\bar{Y}%
_{s}^{\epsilon})ds\right]  dt\nonumber\\
&  =\frac{g(\epsilon)}{\delta}\left(  \int_{0}^{T}\frac{1}{\Delta}\int
_{t}^{t+\Delta}\left[  \mathcal{G}_{\bar{X}_{s}^{\epsilon},\bar{Y}%
_{s}^{\epsilon},u_{s}^{\epsilon}}f_{\ell}(\bar{Y}_{s}^{\epsilon}%
)-\mathcal{G}_{\bar{X}_{t}^{\epsilon},\bar{Y}_{s}^{\epsilon},u_{s}^{\epsilon}%
}f_{\ell}(\bar{Y}_{s}^{\epsilon})\right]  dsdt\right) \nonumber\\
&  \mbox{}+\frac{g(\epsilon)}{\delta}\left(  \int_{0}^{T}\frac{1}{\Delta
}\left[  \int_{t}^{t+\Delta}\mathcal{G}_{\bar{X}_{t}^{\epsilon},\bar{Y}%
_{s}^{\epsilon},u_{s}^{\epsilon}}f_{\ell}(\bar{Y}_{s}^{\epsilon})ds\right]
dt\right) \nonumber\\
&  \mbox{}+\frac{\epsilon g(\epsilon)}{\delta^{2}}\left(  \int_{0}^{T}\frac
{1}{\Delta}\left[  \int_{t}^{t+\Delta}\left[  \mathcal{L}_{\bar{X}%
_{s}^{\epsilon}}^{1}f_{\ell}(\bar{Y}_{s}^{\epsilon})-\mathcal{L}_{\bar{X}%
_{t}^{\epsilon}}^{1}f_{\ell}(\bar{Y}_{s}^{\epsilon})\right]  ds\right]
dt\right) \nonumber\\
&  \mbox{}+\frac{\epsilon g(\epsilon)}{\delta^{2}}\left(  \int_{0}^{T}\frac
{1}{\Delta}\left[  \int_{t}^{t+\Delta}\left[  \mathcal{L}_{\bar{X}%
_{t}^{\epsilon}}^{1}f_{\ell}(\bar{Y}_{s}^{\epsilon})\right]  ds\right]
dt\right) \nonumber\\
&  =\frac{\delta}{\epsilon}\left(  \int_{0}^{T}\frac{1}{\Delta}\left[
\int_{t}^{t+\Delta}\left[  \mathcal{G}_{\bar{X}_{s}^{\epsilon},\bar{Y}%
_{s}^{\epsilon},u_{s}^{\epsilon}}f_{\ell}(\bar{Y}_{s}^{\epsilon}%
)-\mathcal{G}_{\bar{X}_{t}^{\epsilon},\bar{Y}_{s}^{\epsilon},u_{s}^{\epsilon}%
}f_{\ell}(\bar{Y}_{s}^{\epsilon})\right]  ds\right]  dt\right) \nonumber\\
&  \mbox{}+\frac{\delta}{\epsilon}\left(  \int_{\mathcal{Z}\times
\mathcal{Y}\times\lbrack0,T]}\mathcal{G}_{\bar{X}_{t}^{\epsilon},y,z}f_{\ell
}(y)\mathrm{P}^{\epsilon,\Delta}(dzdydt)\right) \nonumber\\
&  \mbox{}+\int_{0}^{T}\frac{1}{\Delta}\left[  \int_{t}^{t+\Delta}\left[
\mathcal{L}_{\bar{X}_{s}^{\epsilon}}^{1}f_{\ell}(\bar{Y}_{s}^{\epsilon
})-\mathcal{L}_{\bar{X}_{t}^{\epsilon}}^{1}f_{\ell}(\bar{Y}_{s}^{\epsilon
})\right]  ds\right]  dt\nonumber\\
&  \mbox{}+\int_{\mathcal{Z}\times\mathcal{Y}\times\lbrack0,T]}\mathcal{L}%
_{\bar{X}_{t}^{\epsilon}}^{1}f_{\ell}(y)\mathrm{P}^{\epsilon,\Delta}(dzdydt).
\label{Eq:MartingaleProperty_1}%
\end{align}
First consider the left hand side of (\ref{Eq:MartingaleProperty_1}). Since
$f_{\ell}$ is bounded $g(\epsilon)\left[  f_{\ell}(\bar{Y}_{T}^{\epsilon
})-f_{\ell}(\bar{Y}_{0}^{\epsilon})\right]  $ converges to zero uniformly. We
claim that
\begin{equation}
g(\epsilon)M_{T}^{\epsilon}\rightarrow0\text{ in probability.}
\label{Eq:MartingaleProperty}%
\end{equation}
Indeed, since $\sigma$ is uniformly bounded $\mathbb{E}_{x_{0}}\left[
M_{T}^{\epsilon}\right]  ^{2}$ is bounded above by a constant times
$\epsilon/\delta^{2}=$ $1/g(\epsilon)$, and so (\ref{Eq:MartingaleProperty})
also follows from $g(\epsilon)\downarrow0$. Finally, we claim that
$g(\epsilon)e_{T}^{\epsilon}\rightarrow0$ in probability. Using Condition
\ref{A:Assumption1}, for some constants $C_{1}$ and $C_{2}$
\begin{align*}
g(\epsilon)\int_{0}^{\Delta}|\mathcal{A}_{u_{s}^{\epsilon},\bar{X}%
_{s}^{\epsilon}}^{\epsilon}f_{\ell}(\bar{Y}_{s}^{\epsilon})|ds  &  \leq
g(\epsilon)C_{1}\Delta\frac{\epsilon}{\delta^{2}}+g(\epsilon)C_{2}\frac
{1}{\delta}\int_{0}^{\Delta}(1+\left\Vert u_{s}^{\epsilon}\right\Vert )ds\\
&  \leq C_{1}\Delta+C_{2}\frac{\delta}{\epsilon}\left[  \frac{3}{2}%
\Delta+\frac{1}{2}\int_{0}^{1}\left\Vert u_{s}^{\epsilon}\right\Vert
^{2}ds\right]  ,
\end{align*}
and hence the left hand side tends to zero in probability by
(\ref{A:UniformlyAdmissibleControls}) and since $\Delta\downarrow
0,\delta/\epsilon\downarrow0$. The same estimate holds for the second term in
$g(\epsilon)e_{T}^{\epsilon}$, and so the claim follows.

Next consider the right hand side of (\ref{Eq:MartingaleProperty_1}). The
first and the third term in the right hand side of
(\ref{Eq:MartingaleProperty_1}) converge to zero in probability by the
tightness of $\bar{X}^{\epsilon}$, Condition \ref{A:Assumption1},
(\ref{A:UniformlyAdmissibleControls}) and $\delta/\epsilon\downarrow0$. The
second term on the right hand side of (\ref{Eq:MartingaleProperty_1})
converges to zero in probability by the uniform integrability of
$\mathrm{P}^{\epsilon,\Delta}$ and by the fact that $\delta/\epsilon
\downarrow0$. So, it remains to consider the fourth term. Passing to the limit
as $\epsilon\downarrow0$, the previous discussion implies that except on a set
$N_{\ell,T}$ of probability zero,%
\begin{equation}
0=\int_{0}^{T}\int_{\mathcal{Z}\times\mathcal{Y}}\mathcal{L}_{\bar{X}_{t}}%
^{1}f_{\ell}(y)\mathrm{P}(dzdydt)\text{.} \label{Eq:ALimit}%
\end{equation}
Let $N=\cup_{\ell\in\mathbb{N}}\cup_{T\in D}N_{\ell,T}$. Then except on the
set $N$ of probability zero, continuity in $T$ and denseness of $\left\{
f_{\ell},\ell\in\mathbb{N}\right\}  $ imply that (\ref{Eq:ALimit}) holds for
all $T\in\lbrack0,1]$ and all $f\in\mathcal{C}^{2}(\mathcal{Y})$.

It remains to prove that $\mathrm{P}(\mathcal{Z}\times\mathcal{Y}\times
\lbrack0,t])=t$ for every $t\in\lbrack0,1]$. Using the fact that the analogous
property holds at the prelimit level, $\mathrm{P}(\mathcal{Z}\times
\mathcal{Y}\times\left\{  t\right\}  )=$ $0$ and the continuity of
$t\rightarrow\mathrm{P}(\mathcal{Z}\times\mathcal{Y}\times\lbrack0,t])$ to
deal with null sets, this property also follows. \hfill$\square$

\subsubsection{Limiting behavior of the CSDE in Regimes 2 and 3.}

\label{SS:ConvergenceRegime2}

In this subsection we prove Theorem \ref{T:MainTheorem1} for $i=2$. The proof
for $i=3$ is similar and thus it is omitted.

\noindent\textit{Proof of Theorem \ref{T:MainTheorem1} for $i=2$.} The proof
follows the same steps as the proof of Theorem \ref{T:MainTheorem1} for $i=1$,
and hence only the differences are outlined. We have
\[
\lambda(x,y,z)=\gamma b(x,y)+c(x,y)+\sigma(x,y)z
\]
and the operator $\bar{\mathcal{A}}_{t}^{\epsilon,\Delta}$ is defined as in
(\ref{Eq:PrelimitOperator}), but with this particular function $\lambda$.

The proof of (\ref{Eq:AccumulationPointsProcessViable}) can be carried out
repeating the corresponding steps of the proof of Theorem \ref{T:MainTheorem1}
for $i=1$. A difference is that one skips the step of applying It\^{o}'s
formula to $\psi_{\ell}$ that satisfies (\ref{Eq:CellProblem2}), since in this
case we do not have an unbounded drift term.

It remains to discuss (\ref{Eq:AccumulationPointsMeasureViable}). Again,
define $\bar{Y}^{\epsilon}=\bar{X}^{\epsilon}/\delta$ and observe that for
$f_{\ell}:\mathcal{Y}\mapsto\mathbb{R}$, $\ell\in\mathbb{N}$ smooth and dense
in $\mathcal{C}(\mathcal{Y})$, $M_{t}^{\epsilon}$ defined by
(\ref{Def:Martingale}) is an $\mathfrak{F}_{t}-$martingale. For any $T>0$,
small $\Delta>0$ and recalling that in this case $g(\epsilon)=\epsilon$,
\begin{align*}
\lefteqn{g(\epsilon)M_{T}^{\epsilon}-g(\epsilon)\left[  f_{\ell}(\bar{Y}%
_{T}^{\epsilon})-f_{\ell}(\bar{Y}_{0}^{\epsilon})\right]  +g(\epsilon
)e_{T}^{\epsilon}}\\
&  \hspace{0.1cm}=\int_{0}^{T}\frac{1}{\Delta}\left[  \int_{t}^{t+\Delta
}\left(  \epsilon\mathcal{A}_{u_{s}^{\epsilon},\bar{X}_{s}^{\epsilon}%
}^{\epsilon}-\gamma\mathcal{L}_{u_{s}^{\epsilon},\bar{X}_{s}^{\epsilon}}%
^{2}\right)  f_{\ell}(\bar{Y}_{s}^{\epsilon})ds\right]  dt+\gamma\int_{0}%
^{T}\frac{1}{\Delta}\left[  \int_{t}^{t+\Delta}\mathcal{L}_{u_{s}^{\epsilon
},\bar{X}_{s}^{\epsilon}}^{2}f_{\ell}(\bar{Y}_{s}^{\epsilon})ds\right]  dt.
\end{align*}
Observing that the operator $\epsilon\mathcal{A}_{z,x}^{\epsilon}$ converges
to the operator $\gamma\mathcal{L}_{z,x}^{2}$, we can argue similarly to the
corresponding part of the proof of Theorem \ref{T:MainTheorem1} for $i=1$ and
conclude that
\[
\int_{0}^{T}\int_{\mathcal{Z}\times\mathcal{Y}}\mathcal{L}_{z,\bar{X}_{t}^{2}%
}^{2}f_{\ell}(y)\mathrm{P}^{2}(dzdydt)=0\text{ w.p.}1.
\]

\section{Laplace principle lower bound and compactness of level sets}

\label{S:LowerBoundLowerSemicontinuity} In this section we prove the Laplace
principle lower bound for Theorem \ref{T:MainTheorem2} and the compactness of
the level sets of the action functional.

\subsection{Laplace principle lower bound}

\label{S:LaplaceLowerBound} For each $\epsilon>0$, let $X^{\epsilon}$ be the
unique strong solution to (\ref{Eq:LDPandA1}). To prove the Laplace principle
lower bound we must show that for all bounded, continuous functions $h$
mapping $\mathcal{C}([0,1];\mathbb{R}^{d})$ into $\mathbb{R}$
\begin{equation*}
\liminf_{\epsilon\downarrow0}-\epsilon\ln\mathbb{E}_{x_{0}}\left[
\exp\left\{  -\frac{h(X^{\epsilon})}{\epsilon}\right\}  \right]  \geq
\inf_{(\phi,\mathrm{P})\in\mathcal{V}}\left[  \frac{1}{2}\int_{\mathcal{Z}%
\times\mathcal{Y}\times\lbrack0,1]}\left\Vert z\right\Vert ^{2}\mathrm{P}%
(dzdydt)+h(\phi)\right]  . \label{Eq:LaplacePrincipleLowerBound}%
\end{equation*}
Of course, it is sufficient to prove the lower limit
(\ref{Eq:LaplacePrincipleLowerBound}) along any subsequence such that
\[
-\epsilon\ln\mathbb{E}_{x_{0}}\left[  \exp\left\{  -\frac{h(X^{\epsilon}%
)}{\epsilon}\right\}  \right]
\]
converges. Such a subsequence exists since $|-\epsilon\ln\mathbb{E}_{x_{0}%
}\left[  \exp\left\{  -h(X^{\epsilon})/\epsilon\right\}  \right]
|\leq\left\Vert h\right\Vert _{\infty}$.

According to Theorem \ref{T:RepresentationTheorem}, there exists a family of
controls $\{u^{\epsilon},\epsilon>0\}$ in $\mathcal{A}$ such that for every
$\epsilon>0$
\begin{equation*}
-\epsilon\ln\mathbb{E}_{x_{0}}\left[  \exp\left\{  -\frac{h(X^{\epsilon}%
)}{\epsilon}\right\}  \right]  \geq\mathbb{E}_{x_{0}}\left[  \frac{1}{2}%
\int_{0}^{1}\left\Vert u_{t}^{\epsilon}\right\Vert ^{2}dt+h(\bar{X}^{\epsilon
})\right]  -\epsilon, \label{Eq:LaplacePrincipleLowerBound1}%
\end{equation*}
where the controlled process $\bar{X}^{\epsilon}$ is defined in
(\ref{Eq:LDPandA2}). Note that for each $\epsilon>0$ $\mathbb{E}_{x_{0}}%
\int_{0}^{1}\left\Vert u_{t}^{\epsilon}\right\Vert ^{2}dt\leq4\left\Vert
h\right\Vert _{\infty}+2\epsilon$, and hence if we use this family of controls
and the associated controlled process $\bar{X}^{\epsilon}$ to construct
occupation measures $\mathrm{P}^{\epsilon,\Delta}$ in
(\ref{Def:OccupationMeasures2}), then by Proposition \ref{P:tightness} the
family $\{\bar{X}^{\epsilon},\mathrm{P}^{\epsilon,\Delta},\epsilon>0\}$ is
tight. Thus given any subsequence of $\epsilon>0$ there is a further
subsubsequence for which
\[
(\bar{X}^{\epsilon},\mathrm{P}^{\epsilon,\Delta})\rightarrow(\bar
{X},\mathrm{P})\text{ in distribution}%
\]
with $(\bar{X},\mathrm{P})\in\mathcal{V}$. By Fatou's lemma
\begin{align*}
\liminf_{\epsilon\downarrow0}\left(  -\epsilon\ln\mathbb{E}_{x_{0}}\left[
\exp\left\{  -\frac{h(X^{\epsilon})}{\epsilon}\right\}  \right]  \right)   &
\geq\liminf_{\epsilon\downarrow0}\left(  \mathbb{E}_{x_{0}}\left[  \frac{1}%
{2}\int_{0}^{1}\left\Vert u_{t}^{\epsilon}\right\Vert ^{2}dt+h(\bar
{X}^{\epsilon})\right]  -\epsilon\right) \nonumber\\
&  \geq\liminf_{\epsilon\downarrow0}\left(  \mathbb{E}_{x_{0}}\left[  \frac
{1}{2}\int_{0}^{1}\frac{1}{\Delta}\int_{t}^{t+\Delta}\left\Vert u_{s}%
^{\epsilon}\right\Vert ^{2}dsdt+h(\bar{X}^{\epsilon})\right]  \right)
\nonumber\\
&  =\liminf_{\epsilon\downarrow0}\left(  \mathbb{E}_{x_{0}}\left[  \frac{1}%
{2}\int_{\mathcal{Z}\times\mathcal{Y}\times\lbrack0,1]}\left\Vert z\right\Vert
^{2}\mathrm{P}^{\epsilon,\Delta}(dzdydt)+h(\bar{X}^{\epsilon})\right]  \right)
\nonumber\\
&  \geq\mathbb{E}_{x_{0}}\left[  \frac{1}{2}\int_{\mathcal{Z}\times
\mathcal{Y}\times\lbrack0,1]}\left\Vert z\right\Vert ^{2}\mathrm{P}%
(dzdydt)+h(\bar{X})\right] \nonumber\\
&  \geq\inf_{(\phi,\mathrm{P})\in\mathcal{V}}\left\{  \frac{1}{2}%
\int_{\mathcal{Z}\times\mathcal{Y}\times\lbrack0,1]}\left\Vert z\right\Vert
^{2}\mathrm{P}(dzdydt)+h(\phi)\right\}  .
\label{Eq:LaplacePrincipleLowerBoundGeneral}%
\end{align*}
This concludes the proof of the Laplace principle lower bound.

\subsection{Compactness of level sets}

\label{LowerSemicontinuity}

Consider $S^{i}(\phi)$ as defined by (\ref{Eq:GeneralRateFunction}) and for
notational convenience omit the superscript $i$ since the proof is independent
of the regime under consideration. We want to prove that for each $s<\infty$,
the set
\[
\Phi_{s}=\{\phi\in\mathcal{C}([0,1];\mathbb{R}^{d}):S(\phi)\leq s\}
\]
is a compact subset of $\mathcal{C}([0,1];\mathbb{R}^{d})$. As usual with the
weak convergence approach, the proof is analogous to that of the Laplace
principle lower bound. In Lemma \ref{L:equicontinuity} we show precompactness
of $\Phi_{s}$ and in Lemma \ref{L:lowersemicontinuous} that it is closed.
Together they imply compactness of $\Phi_{s}$.

\begin{lemma}
\label{L:equicontinuity} Fix $K<\infty$ and consider any sequence $\{(\phi
^{n},\mathrm{P}^{n}),n>0\}$ such that for every $n>0$ $(\phi^{n}%
,\mathrm{P}^{n})$ is viable and
\[
\int_{\mathcal{Z}\times\mathcal{Y}\times\lbrack0,1]}\left\Vert z\right\Vert
^{2}\mathrm{P}^{n}(dzdydt)<K.
\]
Then $\{(\phi^{n},\mathrm{P}^{n}),n>0\}$ is precompact.
\end{lemma}

\begin{proof}
For any $\mathrm{P}$ such that $(\phi,\mathrm{P})\in\mathcal{V}$
\begin{align}
|\phi_{t_{2}}-\phi_{t_{1}}|  &  =\left\vert \int_{\mathcal{Z}\times
\mathcal{Y}\times\lbrack t_{1},t_{2}]}\lambda(\phi_{s},y,z)\mathrm{P}%
(dzdyds)\right\vert \nonumber\\
&  \leq C_{0}\left[  |t_{2}-t_{1}|+\sqrt{(t_{2}-t_{1})}\sqrt{\int
_{\mathcal{Z}\times\mathcal{Y}\times\lbrack t_{1},t_{2}]}\left\Vert
z\right\Vert ^{2}\mathrm{P}(dzdydt)}\right]  .\nonumber
\end{align}
This implies the precompactness of $\{\phi^{n},n>0\}$. Precompactness of
$\{\mathrm{P}^{n},n>0\}$ follows from the compactness of $\mathcal{Y}%
\times\lbrack0,1]$ and that $\left\Vert z\right\Vert ^{2}$ is a tightness
function (similarly to Proposition \ref{P:tightness}, part (i)).
\end{proof}

Next, we prove that the limit of a viable pair is also viable.

\begin{lemma}
\label{L:ContinuityMap} Fix $K<\infty$ and consider any convergent sequence
$\{(\phi^{n},\mathrm{P}^{n}),n>0\}$ such that for every $n>0$ $(\phi
^{n},\mathrm{P}^{n})$ is viable and
\begin{equation}
\int_{\mathcal{Z}\times\mathcal{Y}\times\lbrack0,1]}\left\Vert z\right\Vert
^{2}\mathrm{P}^{n}(dzdydt)<K. \label{A:AdmissibleLimitingMeasures}%
\end{equation}
Then $(\phi,\mathrm{P})$ is a viable pair.
\end{lemma}

\begin{proof}
Since $(\phi^{n},\mathrm{P}^{n})$ is viable
\begin{equation}
\phi_{t}^{n}=x_{0}+\int_{\mathcal{Z}\times\mathcal{Y}\times\lbrack0,t]}%
\lambda(\phi_{s}^{n},y,z)\mathrm{P}^{n}(dzdyds) \label{Eq:ContinuityLemma1_1}%
\end{equation}
and
\begin{equation}
\int_{0}^{t}\int_{\mathcal{Z}\times\mathcal{Y}}\mathcal{L}_{z,\phi_{s}^{n}%
}f(y)\mathrm{P}^{n}(dzdyds)=0 \label{Eq:ContinuityLemma1_2}%
\end{equation}
for every $t\in\lbrack0,1]$ and for every $f\in\mathcal{C}^{2}(\mathcal{Y})$.
The function $\lambda(x,y,z)$ and the operator $\mathcal{L}_{z,x}$ are defined
in Definitions \ref{Def:ThreePossibleFunctions} and
\ref{Def:ThreePossibleOperators} respectively.

By Fatou's Lemma $\mathrm{P}$ has a finite second moment in $z$. Moreover,
observe that the function $\lambda(x,y,z)$ and the operator $\mathcal{L}%
_{z,x}$ are continuous in $x$ and $y$ and affine in $z$. Hence by assumption
(\ref{A:AdmissibleLimitingMeasures}) and the convergence $\mathrm{P}%
^{n}\rightarrow\mathrm{P}$ and $\phi^{n}\rightarrow\phi$, $(\phi,\mathrm{P})$
satisfy equation (\ref{Eq:ContinuityLemma1_1}) with $(\phi^{n},\mathrm{P}%
^{n})$ replaced by $(\phi,\mathrm{P})$.

Next we show that (\ref{Eq:ContinuityLemma1_2}) holds with $(\phi
^{n},\mathrm{P}^{n})$ replaced by $(\phi,\mathrm{P})$. Since
(\ref{A:AdmissibleLimitingMeasures}) holds and $\mathrm{P}(\mathcal{Z}%
\times\mathcal{Y}\times\left\{  t\right\}  )=0$, we can send $n \rightarrow
\infty$ in (\ref{Eq:ContinuityLemma1_2}) and obtain
\[
0=\int_{0}^{t}\int_{\mathcal{Z}\times\mathcal{Y}}\mathcal{L}_{z,\phi_{s}%
}f(y)\mathrm{P}(dzdyds).
\]
Finally, it follows from $\mathrm{P}^{n}(\mathcal{Z}\times\mathcal{Y}%
\times\lbrack0,t])=t$ and $\mathrm{P}(\mathcal{Z}\times\mathcal{Y}%
\times\left\{  t\right\}  )=$ $0$ that $\mathrm{P}(\mathcal{Z}\times
\mathcal{Y}\times\lbrack0,t])=t$ for all $t\in\lbrack0,1]$.
\end{proof}

\begin{lemma}
\label{L:lowersemicontinuous} The functional $S(\phi)$ is lower semicontinuous.
\end{lemma}

\begin{proof}
Let us consider a sequence $\phi^{n}$ with limit $\phi$. We want to prove
\[
\liminf_{n\rightarrow\infty}S(\phi^{n})\geq S(\phi).
\]
It suffices to consider the case when $S(\phi^{n})$ has a finite limit, i.e.,
there exists a $M<\infty$ such that $\liminf_{n\rightarrow\infty}S(\phi
^{n})\leq M$.

We recall the definition
\[
S(\phi^{n})=\inf_{(\phi^{n},\mathrm{P}^{n})\in\mathcal{V}_{(\lambda
,\mathcal{L})}}\left[  \frac{1}{2}\int_{\mathcal{Z}\times\mathcal{Y}%
\times\lbrack0,1]}\left\Vert z\right\Vert ^{2}\mathrm{P}^{n}(dzdydt)\right]
.
\]
Hence we can find measures $\{\mathrm{P}^{n},n<\infty\}$ satisfying $(\phi
^{n},\mathrm{P}^{n})\in\mathcal{V}_{(\lambda,\mathcal{L})}$ and
\begin{equation*}
\sup_{n<\infty}\frac{1}{2}\int_{\mathcal{Z}\times\mathcal{Y}\times\lbrack
0,1]}\left\Vert z\right\Vert ^{2}\mathrm{P}^{n}(dzdyds)<M+1,
\label{ConditionForTightness}%
\end{equation*}
and such that
\[
S(\phi^{n})\geq\left[  \frac{1}{2}\int_{\mathcal{Z}\times\mathcal{Y}%
\times\lbrack0,1]}\left\Vert z\right\Vert ^{2}\mathrm{P}^{n}(dzdydt)-\frac
{1}{n}\right]  .
\]
It follows from Lemma \ref{L:equicontinuity} that we can consider a
subsequence along which $(\phi^{n},\mathrm{P}^{n})$ converges to a limit
$(\phi,\mathrm{P})$. By Lemma \ref{L:ContinuityMap} $(\phi,\mathrm{P})$ is
viable. Hence by Fatou's Lemma
\begin{align*}
\liminf_{n\rightarrow\infty}S(\phi^{n})  &  \geq\liminf_{n\rightarrow\infty
}\left[  \frac{1}{2}\int_{\mathcal{Z}\times\mathcal{Y}\times\lbrack
0,1]}\left\Vert z\right\Vert ^{2}\mathrm{P}^{n}(dzdydt)-\frac{1}{n}\right] \\
&  \geq\frac{1}{2}\int_{\mathcal{Z}\times\mathcal{Y}\times\lbrack
0,1]}\left\Vert z\right\Vert ^{2}\mathrm{P}(dzdydt)\\
&  \geq\inf_{(\phi,\mathrm{P})\in\mathcal{V}_{(\lambda,\mathcal{L})}}\left[
\frac{1}{2}\int_{\mathcal{Z}\times\mathcal{Y}\times\lbrack0,1]}\left\Vert
z\right\Vert ^{2}\mathrm{P}(dzdydt)\right] \\
&  =S(\phi),
\end{align*}
which concludes the proof of lower-semicontinuity of $S(\cdot)$.
\end{proof}

\section{Regime 1: Laplace principle upper bound and alternative
representation}

\label{S:LaplacePrincipleRegime1}

In this section we prove the Laplace principle upper bound for Regime 1. We
also prove in Theorem \ref{T:MainTheorem3} that the formula for the rate
function of Theorem \ref{T:MainTheorem2} takes an explicit form which
coincides with the form provided in \cite{FS}. For notational convenience we
drop the superscript $1$ from $\bar{X}^{1}$ and $\mathrm{P}^{1}$.

In each regime the same steps are taken. The rate function on path space
obtained in the proof of the large deviation upper bound is defined in terms a
viable pair through (\ref{Eq:GeneralRateFunction}). What differs between
regimes are the forms that $\lambda_{i}$ and $\mathcal{L}^{i}$ take. In each
case, we consider for the limit variational problem in the Laplace principle a
nearly optimal pair $(\phi,\mathrm{P})$. Using the notion of viability
appropriate to the particular regime, we examine the constraints that link
$\phi$ and $\mathrm{P}$. The last step is to construct, based on these
constraints, a control for the prelimit representation that will lead to
controls and controlled processes that will converge to the cost associated
with $\mathrm{P}$ and $\phi$, respectively. This construction is subtle in all
regimes, due to the multiscale aspect of the dynamics.

To begin the construction, first observe that one can write
(\ref{Eq:GeneralRateFunction}) in terms of a local rate function, i.e., in the
form%
\begin{equation*}
S^{1}(\phi)=\int_{0}^{1}L_{1}^{r}(\phi_{s},\dot{\phi}_{s})ds.
\label{Eq:GeneralRateFncRegime1}%
\end{equation*}
This follows from the definition of a viable pair by setting
\begin{equation}
L_{1}^{r}(x,\beta)=\inf_{\mathrm{P}\in\mathcal{A}_{x,\beta}^{1,r}}%
\int_{\mathcal{Z}\times\mathcal{Y}}\frac{1}{2}\left\Vert z\right\Vert
^{2}\mathrm{P}(dzdy), \label{Def:ErgodicControlConstraints1_1Relaxed}%
\end{equation}
where
\begin{align*}
\mathcal{A}_{x,\beta}^{1,r}  &  =\left\{  \mathrm{P}\in\mathcal{P}%
(\mathcal{Z}\times\mathcal{Y}):\int_{\mathcal{Z}\times\mathcal{Y}}%
\mathcal{L}_{x}^{1}f(y)\mathrm{P}(dzdy)=0\text{ for all }f\in C^{2}%
(\mathcal{Y})\right. \\
&  \left.  \hspace{2cm}\int_{\mathcal{Z}\times\mathcal{Y}}\left\Vert
z\right\Vert ^{2}\mathrm{P}(dzdy)<\infty\text{ and }\beta=\int_{\mathcal{Z}%
\times\mathcal{Y}}\lambda_{1}(x,y,z)\mathrm{P}(dzdy)\right\}  .
\end{align*}
Note that any measure $\mathrm{P}\in\mathcal{P}(\mathcal{Z}\times\mathcal{Y})$
can be decomposed in the form
\begin{equation}
\mathrm{P}(dzdy)=\eta(dz|y)\mu(dy), \label{Eq:DecompositionOfInvariantMeasure}%
\end{equation}
where $\mu$ is a probability measure on $\mathcal{Y}$ and $\eta$ is a
stochastic kernel on $\mathcal{Z}$ given $\mathcal{Y}$. We refer to this as a
\textquotedblleft relaxed\textquotedblright\ formulation because the control
is characterized as a distribution on $\mathcal{Z}$ (given $x$ and $y $)
rather then as an element of $\mathcal{Z}$. Inserting
(\ref{Eq:DecompositionOfInvariantMeasure}) into
(\ref{Eq:AccumulationPointsMeasureViable}) with $\mathcal{L}_{z,x}%
=\mathcal{L}_{x}^{1}$ from Definition \ref{Def:ThreePossibleOperators}, we get
that for every $f\in\mathcal{C}^{2}(\mathcal{Y})$
\begin{equation}
\int_{\mathcal{Y}}\mathcal{L}_{x}^{1}f(y)\mu(dy)=0.
\label{Eq:AccumulationPointsMeasureRegime1_2}%
\end{equation}
Here we have used the independence of $\mathcal{L}_{x}^{1}$ on the control
variable $z$ to eliminate $\eta$. The nondegeneracy of the diffusion matrix
$\sigma\sigma^{T}$ and (\ref{Eq:AccumulationPointsMeasureRegime1_2}) guarantee
that $\mu(dy)$ is actually the unique invariant measure corresponding to the
operator $\mathcal{L}_{x}^{1}$ with periodic boundary conditions. Naturally,
$\mu(dy)$ implicitly depends on $x$ and was identified in Condition
\ref{A:Assumption2} as $\mu(dy|x)$.

We note that because the cost is convex in $z$ and $\lambda_{1}$ is affine in
$z$, the relaxed control formulation as given in
(\ref{Def:ErgodicControlConstraints1_1Relaxed}) is equivalent to the following
ordinary control formulation of the local rate function:%
\begin{equation}
L_{1}^{o}(x,\beta)=\inf_{(v,\mu)\in\mathcal{A}_{x,\beta}^{1,o}}\frac{1}{2}%
\int_{\mathcal{Y}}\left\Vert v(y)\right\Vert ^{2}\mu(dy),
\label{Def:ErgodicControlConstraints1_1Ordinary}%
\end{equation}
where%
\begin{align*}
\mathcal{A}_{x,\beta}^{1,o}  &  =\left\{  v(\cdot):\mathcal{Y}\mapsto
\mathbb{R}^{d},\mu\in\mathcal{P}(\mathcal{Y})\hspace{0.1cm}:\hspace
{0.1cm}(v,\mu)\text{ satisfy }\int_{\mathcal{Y}}\mathcal{L}_{x}^{1}%
f(y)\mu(dy)=0\right. \\
&  \left.  \hspace{1.8cm}\text{ for all }f\in C^{2}(\mathcal{Y}),\int
_{\mathcal{Y}}\left\Vert v(y)\right\Vert ^{2}\mu(dy)<\infty\text{ and }%
\beta=\int_{\mathcal{Y}}\lambda_{1}(x,y,v(y))\mu(dy)\right\}  .
\end{align*}
The relaxed control formulation turns out to be more convenient when studying
convergence. The fact that $L_{1}^{r}(x,\beta)=L_{1}^{o}(x,\beta)$ follows
from Jensen's inequality and that $\lambda_{1}(x,y,z)$ is affine in $z$. To be
precise, since $(v,\mu)\in\mathcal{A}_{x,\beta}^{1,o}$ induces a
$\mathrm{P}\in\mathcal{A}_{x,\beta}^{1,r}$ via $\mathrm{P}(dzdy)=\delta
_{v(y)}(dz)\mu(dy)$, $L_{1}^{r}(x,\beta)\leq L_{1}^{o}(x,\beta)$. Given
$\mathrm{P}\in\mathcal{A}_{x,\beta}^{1,r}$ we can let $\mu$ be its
$y$-marginal, and then define $v(y)=\int_{\mathcal{Z}}z\eta(dz|y)$, where
$\eta(dz|y)$ is the conditional distribution, so that $(v,\mu)\in
\mathcal{A}_{x,\beta}^{1,o}$. By Jensen's inequality
\[
\int_{\mathcal{Z}\times\mathcal{Y}}\frac{1}{2}\left\Vert z\right\Vert
^{2}\mathrm{P}(dzdy)\geq\int_{\mathcal{Y}}\frac{1}{2}\left\Vert \int
_{\mathcal{Z}}z\eta(dz|y)\right\Vert ^{2}\mu(dy)=\frac{1}{2}\int_{\mathcal{Y}%
}\left\Vert v(y)\right\Vert ^{2}\mu(dy),
\]
and so $L_{1}^{r}(x,\beta)\geq L_{1}^{o}(x,\beta)$. The same will be true for
the analogous quantities in Regimes 2 and 3, though there we will also need to
use that the generator is affine in $z$.

An explicit expression for the local rate function
(\ref{Def:ErgodicControlConstraints1_1Ordinary}) will be given in Theorem
\ref{T:ExplicitSolutionRegime1}. It turns on the following technical lemma
which states a H\"{o}lder inequality for integrals of matrices. The proof of
the lemma is deferred to the end of this section.

\begin{lemma}
\label{HolderInequality} Let $\kappa\in\mathcal{L}^{2}(\mathcal{Y},M_{d\times
d}(\mathbb{R});\mu)$ and $u\in\mathcal{L}^{2}(\mathcal{Y},M_{d\times
1}(\mathbb{R});\mu)$ be matrix and vector valued functions, respectively.
Define
\[
\beta=\int_{\mathcal{Y}}\kappa(y)u(y)\mu(dy)\hspace{0.2cm}\text{ and }%
\hspace{0.2cm}q=\int_{\mathcal{Y}}\kappa(y)\kappa^{T}(y)\mu(dy),
\]
and assume that $q$ is positive definite. Then
\begin{equation*}
\beta^{T}q^{-1}\beta\leq\int_{\mathcal{Y}}\left\Vert u(y)\right\Vert ^{2}%
\mu(dy). \label{HolderGoal}%
\end{equation*}

\end{lemma}

\begin{theorem}
\label{T:ExplicitSolutionRegime1} Under Conditions \ref{A:Assumption1} and
\ref{A:Assumption2}, the infimization problem
(\ref{Def:ErgodicControlConstraints1_1Relaxed}) and hence
(\ref{Def:ErgodicControlConstraints1_1Ordinary}) has the explicit solution%
\[
L_{1}^{o}(x,\beta)=\frac{1}{2}(\beta-r(x))^{T}q^{-1}(x)(\beta-r(x)),
\]
where

\begin{itemize}
\item {$r(x)=\int_{\mathcal{Y}}(I+\frac{\partial\chi}{\partial y}%
)(x,y)c(x,y)\mu(dy|x)$,}

\item {$q(x)=\int_{\mathcal{Y}}(I+\frac{\partial\chi}{\partial y}%
)(x,y)\sigma(x,y)\sigma^{T}(x,y)(I+\frac{\partial\chi}{\partial y}%
)^{T}(x,y)\mu(dy|x),$}
\end{itemize}

and where $\mu(dy|x)$ is the unique invariant measure corresponding to the
operator $\mathcal{L}_{x}^{1}$ and $\chi(x,y)$ is defined by
(\ref{Eq:CellProblem}). The control
\[
v(y)=\bar{u}_{\beta}(x,y)=\sigma^{T}(x,y)\left(  I+\frac{\partial\chi
}{\partial y}(x,y)\right)  ^{T}q^{-1}(x)(\beta-r(x))
\]
attains the infimum in (\ref{Def:ErgodicControlConstraints1_1Ordinary}).
\end{theorem}

\begin{proof}
First observe that for any $v\in\mathcal{A}_{x,\beta}^{1,o}$
\[
\int_{\mathcal{Y}}\left\Vert v(y)\right\Vert ^{2}\mu(dy|x)\geq\left(
\beta-r(x)\right)  ^{T}q^{-1}(x)\left(  \beta-r(x)\right)  .
\]
This can be derived as follows. Any $v\in\mathcal{A}_{x,\beta}^{1,o}$
satisfies
\[
\beta=\int_{\mathcal{Y}}\lambda_{1}(x,y,v(y))\mu(dy|x)=r(x)+\int_{\mathcal{Y}%
}\left(  I+\frac{\partial\chi}{\partial y}\right)  \sigma(x,y)\left(
v(y)\right)  ^{T}\mu(dy|x).
\]
Then treating $x$ as a parameter and applying Lemma \ref{HolderInequality} to
the relation above with $\beta-r(x)$ in place of $\beta$, $\kappa(x,y)=(
I+\frac{\partial\chi}{\partial y}) \sigma(x,y)$ and $u(y)=v(y)$ we immediately
get the claim.

Next we observe that by choosing (with $x$ again treated as a parameter)
\[
v(y)=\bar{u}_{\beta}(x,y)=\sigma^{T}(x,y)\left(  I+\frac{\partial\chi
}{\partial y}(x,y)\right)  ^{T}q^{-1}(x)(\beta-r(x)),
\]
we have
\[
\int_{\mathcal{Y}}\left\Vert \bar{u}_{\beta}(x,y)\right\Vert ^{2}%
\mu(dy|x)=(\beta-r(x))^{T}q^{-1}(x)(\beta-r(x)).
\]
This completes the proof of the theorem.
\end{proof}

Now we have all the ingredients to prove the Laplace principle upper bound and
hence to complete the proof of the LDP for $X^{\epsilon}$ in Regime $1$.

\begin{proof}
[Proof of Laplace principle upper bound for Regime 1]For each $\epsilon>0$,
let $X^{\epsilon}$ be the unique strong solution to (\ref{Eq:LDPandA1}). To
prove the Laplace principle upper bound we must show that for all bounded,
continuous functions $h$ mapping $\mathcal{C}([0,1];\mathbb{R}^{d})$ into
$\mathbb{R}$
\begin{equation*}
\limsup_{\epsilon\downarrow0}-\epsilon\ln\mathbb{E}_{x_{0}}\left[
\exp\left\{  -\frac{h(X^{\epsilon})}{\epsilon}\right\}  \right]  \leq
\inf_{\phi\in\mathcal{C}([0,1];\mathbb{R}^{d})}\left[  S(\phi)+h(\phi)\right]
. \label{Eq:LaplacePrincipleUpperBoundRegime1}%
\end{equation*}
Let $\eta>0$ be given and consider $\psi\in\mathcal{C}([0,1];\mathbb{R}^{d})$
with $\psi_{0}=x_{0}$ such that
\begin{equation}
S(\psi)+h(\psi)\leq\inf_{\phi\in\mathcal{C}([0,1];\mathbb{R}^{d})}\left[
S(\phi)+h(\phi)\right]  +\eta<\infty.
\label{Eq:NearlyOptimalTrajectoryRegime1}%
\end{equation}
Since $h$ is bounded, this implies that $S(\psi)<\infty$, and thus $\psi$ is
absolutely continuous. Theorem \ref{T:ExplicitSolutionRegime1} shows that
$L_{1}^{o}(x,\beta)$ is continuous and finite at each $(x,\beta)\in
\mathbb{R}^{2d}$. By a standard mollification argument we can further assume
that $\dot{\psi}$ is piecewise continuous (see for example Subsection $6.5$ of
\cite{DupuisEllis}). Given this particular function $\psi$ define
\begin{equation*}
\bar{u}(t,x,y)=\sigma^{T}(x,y)\left(  I+\frac{\partial\chi}{\partial
y}(x,y)\right)  ^{T}q^{-1}(x)(\dot{\psi}_{t}-r(x)),
\label{Def:ParticularControlRegime1_1}%
\end{equation*}
where $\chi$ satisfies (\ref{Eq:CellProblem}). Clearly, $\bar{u}(t,x,y)$ is
periodic in $y$. Lastly, we define a control in (partial) feedback form by%
\begin{equation*}
\bar{u}^{\epsilon}(t)=\bar{u}\left(  t,X_{t}^{\epsilon},\frac{X_{t}^{\epsilon
}}{\delta}\right)  . \label{Def:ParticularControlRegime1_2}%
\end{equation*}
Then standard homogenization theory for locally periodic diffusions and the
fact that the invariant measure $\mu(\cdot|x)$ is continuous as a function of
$x$ (see for example Chapter $3$, Section $4.6$ of \cite{BLP}) imply the following:

\begin{enumerate}
\item {$\bar{X}^{\epsilon}\overset{\mathcal{D}}{\rightarrow}\bar{X}$, where
w.p.1
\begin{align*}
\bar{X}_{t}  &  =x_{0}+\int_{0}^{t}r(\bar{X}_{s})ds+\int_{0}^{t}\left[
\int_{\mathcal{Y}}\left(  I+\frac{\partial\chi}{\partial y}(\bar{X}%
_{s},y)\right)  \sigma(\bar{X}_{s},y)\bar{u}\left(  s,\bar{X}_{s},y\right)
\mu(dy|\bar{X}_{s})\right]  ds\\
&  =x_{0}+\int_{0}^{t}r(\bar{X}_{s})ds+\int_{0}^{t}\left[  \int_{\mathcal{Y}%
}\left(  I+\frac{\partial\chi}{\partial y}\right)  \sigma\sigma^{T}\left(
I+\frac{\partial\chi}{\partial y}\right)  ^{T}\mu(dy|\bar{X}_{s})\right]
q^{-1}(\bar{X}_{s})\left(  \dot{\psi}_{s}-r(\bar{X}_{s})\right)  ds\\
&  =x_{0}+\int_{0}^{t}r(\bar{X}_{s})ds+\int_{0}^{t}q(\bar{X}_{s})q^{-1}%
(\bar{X}_{s})\left(  \dot{\psi}_{s}-r(\bar{X}_{s})\right)  ds\\
&  =x_{0}+\int_{0}^{t}\dot{\psi}_{s}ds\\
&  =\psi_{t},
\end{align*}
}

\item {the cost satisfies
\begin{equation}
\mathbb{E}_{x_{0}}\left(  \frac{1}{2}\int_{0}^{1}\left\Vert \bar{u}%
_{s}^{\epsilon}\right\Vert ^{2}ds-\frac{1}{2}\int_{0}^{1}\int_{\mathcal{Y}%
}\left\Vert \bar{u}(s,\bar{X}_{s},y)\right\Vert ^{2}\mu(dy|\bar{X}%
_{s})ds\right)  ^{2}\rightarrow0,\text{ as }\epsilon\downarrow0.
\label{Eq:LimitingCostRegime1_1}%
\end{equation}
}
\end{enumerate}

Theorem \ref{T:ExplicitSolutionRegime1} then implies that
\begin{equation}
\mathbb{E}_{x_{0}}\int_{0}^{1}\int_{\mathcal{Y}}\left\Vert \bar{u}(s,\bar
{X}_{s},y)\right\Vert ^{2}\mu(dy|\bar{X}_{s})ds=\mathbb{E}_{x_{0}}S(\bar
{X})=S(\psi). \label{Eq:LimitingCostRegime1_1b}%
\end{equation}
Thus
\begin{align*}
\limsup_{\epsilon\downarrow0}-\epsilon\ln\mathbb{E}_{x_{0}}\left[
\exp\left\{  -\frac{h(X^{\epsilon})}{\epsilon}\right\}  \right]   &
=\limsup_{\epsilon\downarrow0}\inf_{u\in\mathcal{A}}\mathbb{E}_{x_{0}}\left[
\frac{1}{2}\int_{0}^{1}\left\Vert u_{t}\right\Vert ^{2}dt+h(\bar{X}^{\epsilon
})\right] \\
&  \leq\limsup_{\epsilon\downarrow0}\mathbb{E}_{x_{0}}\left[  \frac{1}{2}%
\int_{0}^{1}\left\Vert \bar{u}_{t}^{\epsilon}\right\Vert ^{2}dt+h(\bar
{X}^{\epsilon})\right] \\
&  \leq\mathbb{E}_{x_{0}}\left[  \frac{1}{2}\int_{0}^{1}\int_{\mathcal{Y}%
}\left\Vert \bar{u}(s,\bar{X}_{s},y)\right\Vert ^{2}\mu(dy|\bar{X}%
_{s})dyds+h(\bar{X})\right] \\
&  =\left[  S(\psi)+h(\psi)\right] \\
&  \leq\inf_{\phi\in\mathcal{C}([0,1];\mathbb{R}^{d})}\left[  S(\phi
)+h(\phi)\right]  +\eta.
\end{align*}
Line $1$ follows from the representation Theorem \ref{T:RepresentationTheorem}%
. Line $2$ follows from the choice of a particular control. Line $3$ follows
from (\ref{Eq:LimitingCostRegime1_1}) and the continuity of $h$. Line $4$
follows from (\ref{Eq:LimitingCostRegime1_1b}) and from the fact that $\bar
{X}_{t}=\psi_{t}$. Lastly, line $5$ follows from
(\ref{Eq:NearlyOptimalTrajectoryRegime1}). Since $\eta>0$ is arbitrary, the
upper bound is proved.
\end{proof}

In fact, the considerations above allow us to derive an explicit
representation formula for the rate function in Regime 1. We summarize the
results in the following theorem.

\begin{theorem}
\label{T:MainTheorem3} Let $\{X^{\epsilon},\epsilon>0\}$ be the unique strong
solution to (\ref{Eq:LDPandA1}) and consider Regime $1$. Under Conditions
\ref{A:Assumption1} and \ref{A:Assumption2}, $\{X^{\epsilon},\epsilon>0\}$
satisfies a large deviations principle with rate function
\begin{equation*}
S(\phi)=%
\begin{cases}
\frac{1}{2}\int_{0}^{1}(\dot{\phi}_{s}-r(\phi_{s}))^{T}q^{-1}(\phi_{s}%
)(\dot{\phi}_{s}-r(\phi_{s}))ds & \text{if }\phi\in\mathcal{C}%
([0,1];\mathbb{R}^{d})\text{ is absolutely continuous}\\
+\infty & \text{otherwise.}%
\end{cases}
\label{ActionFunctional1}%
\end{equation*}

\end{theorem}

We conclude with the proof of Lemma \ref{HolderInequality}.

\begin{proof}
[Proof of Lemma \ref{HolderInequality}]
Since $q$ is positive definite and symmetric one can write%
\[
q^{-1}=W^{T}W,
\]
where $W$ is an invertible matrix. It follows that
\[
\beta^{T}q^{-1}\beta=\left\Vert W\beta\right\Vert ^{2}.
\]
Without loss of generality we can assume
\[
\int_{\mathcal{Y}}\left\Vert u(y)\right\Vert ^{2}\mu(dy)=1.
\]
By the Cauchy-Schwartz inequality in $\mathbb{R}^{d}$ we have
\begin{align*}
\left\Vert W\beta\right\Vert ^{2}  &  =\left\langle W\beta,W\int_{\mathcal{Y}%
}\kappa(y)u(y)\mu(dy)\right\rangle \\
&  =\int_{\mathcal{Y}}\left\langle u(y),\kappa^{T}(y)W^{T}W\beta\right\rangle
\mu(dy)\\
&  \leq\sqrt{\int_{\mathcal{Y}}\left\Vert u\right\Vert ^{2}\mu(dy)}\sqrt
{\int_{\mathcal{Y}}\left\Vert \kappa^{T}(y)W^{T}W\beta\right\Vert ^{2}\mu
(dy)}\\
&  =\sqrt{\int_{\mathcal{Y}}\left\Vert \kappa^{T}(y)W^{T}W\beta\right\Vert
^{2}\mu(dy)}\\
&  =\sqrt{\beta^{T}W^{T}W\left[  \int_{\mathcal{Y}}\kappa(y)\kappa^{T}%
(y)\mu(dy)\right]  W^{T}W\beta}\\
&  =\sqrt{\beta^{T}W^{T}W\beta}\\
&  =\left\Vert W\beta\right\Vert .
\end{align*}
If $\left\Vert W\beta\right\Vert =0$, then the result holds automatically. If
$\left\Vert W\beta\right\Vert \neq0$ then we get $\left\Vert W\beta\right\Vert
\leq1$, which proves the result.
\end{proof}

\subsection{Example}

\label{SS:Examples} In this subsection we consider an example. A particular
model of interest is the first order Langevin equation
\begin{equation}
dX_{t}^{\epsilon}=\left[  -\frac{\epsilon}{\delta}\nabla Q\left(  \frac
{X_{t}^{\epsilon}}{\delta}\right)  -\nabla V\left(  X_{t}^{\epsilon}\right)
\right]  dt+\sqrt{\epsilon}\sqrt{2D}dW_{t},\hspace{0.2cm}X_{0}^{\epsilon
}=x_{0}, \label{Eq:LangevinEquation2}%
\end{equation}
where $2D$ is a diffusion constant and the two-scale potential is composed by
a large-scale part, $V(x)$, and a fluctuating part, $\epsilon Q(x/\delta) $.
An example of such a potential is given in Figure 1.

\begin{figure}[th]
\label{F:Figure1}
\par
\begin{center}
\includegraphics[scale=0.4, width=8 cm, height=6 cm]{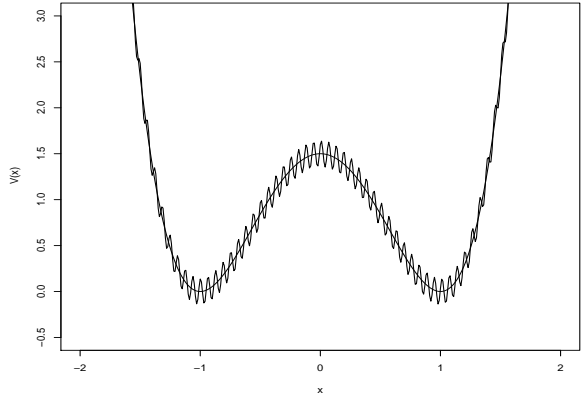}
\end{center}
\caption{ $V^{\epsilon}(x,\frac{x}{\delta})=\epsilon\left(  \cos(\frac
{x}{\delta})+\sin(\frac{x}{\delta})\right)  +\frac{3}{2}(x^{2}-1)^{2}$ and
$V(x)=\frac{3}{2}(x^{2}-1)^{2}$ with $\epsilon=0.1$ and $\delta=0.01$.}%
\end{figure}

To connect to our notation let $b(x,y)=-\nabla Q(y)$ and $c(x,y)=-\nabla V(x)
$, and suppose we consider Regime 1. In this case there is an explicit formula
for the invariant density $\mu(y)$, which is the Gibbs distribution
\[
\mu(y)=\frac{1}{Z}e^{-\frac{Q(y)}{D}},\hspace{0.2cm}Z=\int_{\mathcal{Y}%
}e^{-\frac{Q(y)}{D}}dy.
\]
Moreover, it is easy to see that the centering Condition \ref{A:Assumption2}
holds.

When we have a separable fluctuating part, i.e. $Q(y_{1},y_{2},\ldots
,y_{d})=Q_{1}(y_{1})+Q_{2}(y_{2})+\cdots+Q_{d}(y_{d})$, everything can be
calculated explicitly. We summarize the results in the following corollary.

\hspace{1cm}

\begin{corollary}
\label{C:MainCorollary2} Let $\{X^{\epsilon},\epsilon>0\}$ be the unique
strong solution to (\ref{Eq:LangevinEquation2}). Assume $Q(y_{1},y_{2}%
,\cdots,y_{d})=Q_{1}(y_{1})+Q_{2}(y_{2})+\cdots+Q_{d}(y_{d})$ and consider
Regime $1$. Under Condition \ref{A:Assumption1}, $\{X^{\epsilon},\epsilon>0\}$
satisfies a large deviations principle with rate function
\begin{equation*}
S(\phi)=%
\begin{cases}
\frac{1}{2}\int_{0}^{1}(\dot{\phi}_{s}-r(\phi_{s}))^{T}q^{-1}(\dot{\phi}%
_{s}-r(\phi_{s}))ds & \text{if }\phi\in\mathcal{C}([0,1];\mathbb{R}^{d})\text{
is absolutely continuous }\\
+\infty & \text{otherwise, }%
\end{cases}
\label{ActionFunctional1_2}%
\end{equation*}
where
\[
r(x)=-\Theta\nabla V(x),\hspace{0.2cm}q=2D\Theta,\hspace{0.2cm}\Theta
=\text{diag}\left[  \frac{1}{Z_{1}\hat{Z}_{1}},\cdots,\frac{1}{Z_{d}\hat
{Z}_{d}}\right]
\]
and for $i=1,2,\ldots,d$
\[
Z_{i}=\int_{\mathbb{T}}e^{-\frac{Q_{i}(y_{i})}{D}}dy_{i},\hspace{0.2cm}\hat
{Z}_{i}=\int_{\mathbb{T}}e^{\frac{Q_{i}(y_{i})}{D}}dy_{i}.
\]

\end{corollary}

Observing the effective diffusivity matrix $q$ in Corollary
\ref{C:MainCorollary2}, we see that the diagonal elements of $q$ are always
smaller than the corresponding diagonal elements of the original one. In the
original multiscale problem there are many small energy barriers. These are
not captured by the homogenized potential and hence must be accounted for in
the homogenized process, and thus the trapping from the many local minima is
responsible for the reduction of the diffusion coefficient.

\section{Regime 2: Laplace principle upper bound and alternative
representation.}

\label{S:LaplacePrincipleRegime2}

In this section we prove the Laplace principle upper bound for Regime 2. We
need several auxiliary results that will be proven in Subsection
\ref{S:BoundedOptimalControlRegime2}. For notational convenience we drop the
superscript $2$ from $\bar{X}^{2}$ and $\mathrm{P}^{2}$.

As was done for Regime $1$ we can define the relaxed and ordinary control
formulations of the local rate function, $L_{2}^{r}(x,\beta)$ and $L_{2}%
^{o}(x,\beta)$, by considering $\lambda_{2}$ and $\mathcal{L}_{z,x}^{2}$ in
place of $\lambda_{1}$ and $\mathcal{L}_{x}^{1}$. For the same reasons as in
Section \ref{S:LaplacePrincipleRegime1} (but also using that $\mathcal{L}%
_{z,x}^{2}$ is affine in $z$), these two expressions coincide. The key
difference between this case and the last is that $\mathcal{L}_{z,x}^{2}$
depends on $z$, while $\mathcal{L}_{x}^{1}$ did not. This means that relations
between the elements of a viable pair are more complex, and in particular that
the joint distribution of the control $z$ and fast variable $y$ is important.

Similarly to what was done in Regime $1$, the limiting occupation measure
$\mathrm{P}\in\mathcal{P}(\mathcal{Z}\times\mathcal{Y}\times\lbrack0,1])$ can
be decomposed as stochastic kernels in the form
\begin{equation*}
\mathrm{P}(dzdydt)=\eta(dz|y,t)\mu(dy|t)dt.
\label{Eq:DecompositionOfInvariantMeasureRegime2}%
\end{equation*}
Moreover, by Theorem \ref{T:MainTheorem1} we have $(\bar{X},\mathrm{P}%
)\in\mathcal{V}_{\lambda_{2},\mathcal{L}^{2}}$. We will use that both
$\lambda_{2}$ and $\mathcal{L}_{z,x}^{2}$ are affine in $z$. If
$v(t,y):[0,1]\times\mathcal{Y}\mapsto\mathbb{R}^{d}$ is defined by
\[
v(t,y)=\int_{\mathcal{Z}}z\eta(dz|y,t),
\]
then by viability $\bar{X}_{t}$ satisfies%
\begin{equation*}
\bar{X}_{t}=x_{0}+\int_{0}^{t}\left[  \int_{\mathcal{Y}}\left(  \gamma
b(\bar{X}_{s},y)+c(\bar{X}_{s},y)+\sigma(\bar{X}_{s},y)v(s,y)\right)
\mu(dy|s)\right]  ds \label{Eq:AccumulationPointsProcessRegime2_1a}%
\end{equation*}
where $\mu$ is such that for all $f\in\mathcal{C}^{2}(\mathcal{Y})$ and
$t\in\lbrack0,1]$%
\begin{equation*}
\int_{0}^{t}\int_{\mathcal{Y}}\mathcal{L}_{v(s,y),\bar{X}_{s}}^{2}%
f(y)\mu(dy|s)=0. \label{Eq:AccumulationPointsProcessRegime2_2a}%
\end{equation*}

\noindent\textit{Proof of Laplace principle upper bound for Regime 2.}

We need to prove that
\[
\limsup_{\epsilon\downarrow0}-\epsilon\ln\mathbb{E}_{x_{0}}\left[
\exp\left\{  -\frac{h(X^{\epsilon})}{\epsilon}\right\}  \right]  \leq
\inf_{\phi\in\mathcal{C}([0,1];\mathbb{R}^{d})}\left[  S(\phi)+h(\phi)\right]
.
\]
For given $\eta>0$ we can find $\psi\in\mathcal{C}([0,1];\mathbb{R}^{d})$ with
$\psi_{0}=x_{0}$ such that
\begin{equation}
S(\psi)+h(\psi)\leq\inf_{\phi\in\mathcal{C}([0,1];\mathbb{R}^{d})}\left[
S(\phi)+h(\phi)\right]  +\eta<\infty.
\label{Eq:NearlyOPtimalTrajectoryRegime2}%
\end{equation}
Since $h$ is bounded, this implies that $S(\psi)<\infty$, and thus $\psi$ is
absolutely continuous.

Let
\begin{align*}
\mathcal{A}_{x,\beta}^{2,o}  &  =\left\{  v(\cdot):\mathcal{Y}\mapsto
\mathbb{R}^{d},\mu\in\mathcal{P}(\mathcal{Y})\hspace{0.1cm}:\hspace
{0.1cm}(v,\mu)\text{ satisfy }\int_{\mathcal{Y}}\mathcal{L}_{v(y),x}%
^{2}f(y)\mu(dy)=0\right. \\
&  \left.  \hspace{1.8cm}\text{ for all }f\in C^{2}(\mathcal{Y}),
\int_{\mathcal{Y}}\left\Vert v(y)\right\Vert ^{2}\mu(dy)<\infty\text{ and
}\beta=\int_{\mathcal{Y}}\lambda_{2}(x,y,v(y))\mu(dy)\right\}  .
\end{align*}
Then the ordinary control formulation of the local rate function is
\begin{equation}
L_{2}^{o}(x,\beta)=\inf_{(v,\mu)\in\mathcal{A}_{x,\beta}^{2,o}}\left\{
\frac{1}{2}\int_{\mathcal{Y}}\left\Vert v(y)\right\Vert ^{2}\mu(dy)\right\}  .
\label{Def:LocalRateFunctionRegime2}%
\end{equation}
Calling this an \textquotedblleft ordinary control
formulation\textquotedblright\ is perhaps a bit misleading. Invariant measures
are in general characterized by equations of the form
\begin{equation}
\int_{\mathcal{Y}}\mathcal{L}_{v(y),x}^{2}f(y)\mu(dy)=0
\label{Eq:AccumulationPointsProcessRegime2_2b}%
\end{equation}
where $v(\cdot):\mathcal{Y}\mapsto\mathcal{Z}$ plays the role of a feedback
control. In the definition of $\mathcal{A}_{x,\beta}^{2,o}$ no claim is made
that $\mu$ is an invariant distribution for any controlled dynamics. [This was
not an issue in Regime 1 since $\mathcal{L}_{x}^{1}$ did not depend on $z$.
Hence there was only one invariant distribution that did not depend in any way
on the control.] In fact for some choices of $v$ it may be difficulty to argue
that an invariant distribution corresponding to $\mathcal{L}_{v(y),x}^{2}$
exists. However, we will use results from \cite{KurtzStockbridge} that allow
us to represent $L_{2}^{o}(x,\beta)$ in terms of the average cost of an
ergodic control problem for which the Bellman equation has a classical sense
solution. This will lead to a control $v$ that is bounded and Lipschitz
continuous, and hence for the corresponding controlled diffusion there will be
a unique invariant distribution $\mu$ such that the pair satisfy
(\ref{Eq:AccumulationPointsProcessRegime2_2b}).

By Theorem \ref{T:LocalRateFunctionRegime2} below, $L_{2}^{o}(x,\beta)$ is
continuous and finite at each $(x,\beta)\in\mathbb{R}^{2d}$. Thus, by a
standard mollification argument, we can further assume that $\dot{\psi}$ is
piecewise constant (see for example Subsection $6.5$ in \cite{DupuisEllis}).
Theorem \ref{T:BoundedOptimalControlRegime2} below implies that there is
$\bar{u}(t,x,y)$ that is bounded, continuous in $x$ and Lipschitz continuous
$y$, and piecewise constant in $t$ and which satisfies
\begin{equation}
\bar{u}(t,x,\cdot)\in\text{argmin}_{v}\left\{  \frac{1}{2}\int_{\mathcal{Y}%
}\left\Vert v(y)\right\Vert ^{2}\mu(dy)\hspace{0.1cm}:(v,\mu)\in
\mathcal{A}_{x,\dot{\psi}_{t}}^{2,o}\right\}  .
\label{Eq:OptimalControlRegime2_1}%
\end{equation}

As remarked previously for this particular control $\bar{u}$ the invariant
measure corresponding to the operator $\mathcal{L}_{\bar{u},x}^{2}$ is unique
and will be denoted by $\bar{\mu}_{\bar{u}}(dy)$. The control used in the
large deviation problem (in feedback form) is then%
\[
\bar{u}^{\epsilon}(t)=\bar{u}\left(  t,\bar{X}_{t}^{\epsilon},\frac{\bar
{X}_{t}^{\epsilon}}{\delta}\right)  .
\]
Since $\sigma\sigma^{T}$ is uniformly nondegenerate and Lipschitz continuous
and since $\bar{u}$ is continuous in $x$ and $y$, a strong solution to
(\ref{Eq:LDPandA2}) exists. By standard averaging theory and the fact that
$\bar{\mu}_{\bar{u}(t,x,\cdot)}(\cdot)$ is continuous in $x$ (Theorem
\ref{T:BoundedOptimalControlRegime2}) and piecewise continuous in $t$ we have
that $\bar{X}^{\epsilon}\overset{\mathcal{D}}{\rightarrow}\bar{X}$, where
\begin{equation*}
\bar{X}_{t}=x_{0}+\int_{0}^{t}\int_{\mathcal{Y}}\lambda_{2}\left(  \bar{X}%
_{s},y,\bar{u}(s,\bar{X}_{s},y)\right)  \bar{\mu}_{\bar{u}(s,\bar{X}_{s}%
,\cdot)}(dy)ds.\label{Eq:LimitingXReg2}%
\end{equation*}
Since (\ref{Eq:OptimalControlRegime2_1}) holds we get, for $\psi$ such that
$\psi_{0}=x_{0}$,
\[
\bar{X}_{t}=x_{0}+\int_{0}^{t}\dot{\psi}_{s}ds=\psi_{t}\hspace{0.2cm}\text{
for any }t\in\lbrack0,1]\text{, w.p.}1.
\]
Taking into account the above facts, we have the following chain of
inequalities:
\begin{align}
\limsup_{\epsilon\downarrow0}\left[  -\epsilon\ln\mathbb{E}_{x_{0}}\left[
\exp\left\{  -\frac{h(X^{\epsilon})}{\epsilon}\right\}  \right]  \right]   &
=\limsup_{\epsilon\downarrow0}\inf_{u}\mathbb{E}_{x_{0}}\left[  \frac{1}%
{2}\int_{0}^{1}\left\Vert u_{t}\right\Vert ^{2}dt+h(\bar{X}^{\epsilon
})\right]  \nonumber\\
&  \leq\limsup_{\epsilon\downarrow0}\mathbb{E}_{x_{0}}\left[  \frac{1}{2}%
\int_{0}^{1}\left\Vert \bar{u}^{\epsilon}(t)\right\Vert ^{2}dt+h(\bar
{X}^{\epsilon})\right]  \nonumber\\
&  =\mathbb{E}_{x_{0}}\left[  \frac{1}{2}\int_{0}^{1}\int_{\mathcal{Y}%
}\left\Vert \bar{u}(t,\bar{X}_{t},y)\right\Vert ^{2}\bar{\mu}_{\bar{u}%
(t,\bar{X}_{t},\cdot)}(dy)dt+h(\bar{X})\right]  \nonumber\\
&  =\mathbb{E}_{x_{0}}\left[  S(\bar{X})+h(\bar{X})\right]  \nonumber\\
&  =S(\psi)+h(\psi)\nonumber\\
&  \leq\inf_{\phi\in\mathcal{C}([0,1];\mathbb{R}^{d})}\left[  S(\phi
)+h(\phi)\right]  +\eta.\nonumber
\end{align}
Line $1$ follows from the representation Theorem \ref{T:RepresentationTheorem}%
. Line $2$ follows from the choice of the particular control. Line $3$ follows
from the definition of the control by the minimization problem above and the
continuity of $h$. Line $4$ follows from the definition of $S$. Line 6 is from
(\ref{Eq:NearlyOPtimalTrajectoryRegime2}). Finally, since $\eta$ is arbitrary,
we are done. \medskip\hfill$\square$

In fact, the considerations above allow us to derive an alternative
representation formula for the rate function in Regime 2. We summarize the
results in the following theorem.

\begin{theorem}
\label{T:MainTheorem4Regime2} Let $\{X^{\epsilon}, \epsilon>0\}$ be the unique
strong solution to (\ref{Eq:LDPandA1}) such that Condition \ref{A:Assumption1}
holds and assume that we are considering Regime $2$. Then $\{X^{\epsilon},
\epsilon>0\}$ satisfies a large deviations principle with rate function
\begin{align}
S(\phi)=
\begin{cases}
\int_{0}^{1}L_{2}^{o}(\phi_{s},\dot{\phi}_{s})ds & \text{if }\phi
\in\mathcal{C}([0,1];\mathbb{R}^{d}) \text{ is absolutely continuous}\\
+\infty & \text{otherwise }.
\end{cases}
\nonumber
\end{align}

\end{theorem}

\subsection{Properties of the local rate function and of the optimal control
for Regime $2$.}

\label{S:BoundedOptimalControlRegime2} In this section we study the local rate
function $L_{2}^{o}(x,\beta)=\inf_{(v,\mu)\in\mathcal{A}_{x,\beta}^{2,o}%
}\{\frac{1}{2}\int_{\mathcal{Y}}\left\Vert v(y)\right\Vert ^{2}\mu(dy)\}$. The
main theorems of this section are the following two.

\begin{theorem}
\label{T:BoundedOptimalControlRegime2} Assume Condition \ref{A:Assumption1}.
Then there is a pair $(\bar{u},\bar{\mu})$ that achieves the infimum in the
definition of the local rate function such that $\bar{u}=\bar{u}_{\beta}(x,y)$
is, for each fixed $\beta\in\mathbb{R}^{d}$, continuous in $x$, Lipschitz
continuous in $y$ and measurable in $(x,y,\beta)$. Moreover, $\bar{\mu
}(dy)=\bar{\mu}_{\bar{u}}(dy|x)$ is the unique invariant measure corresponding
to the operator $\mathcal{L}_{\bar{u}_{\beta}(x,y),x}^{2}$ and it is weakly
continuous as a function of $x$.
\end{theorem}

\begin{theorem}
\label{T:LocalRateFunctionRegime2} Assume Condition \ref{A:Assumption1}. Then,
the local rate function $L^{o}_{2}(x,\beta)$ is finite, continuous at each
$(x,\beta)\in\mathbb{R}^{2d}$ and differentiable with respect to $\beta$.
\end{theorem}

The proof of these theorems will be given in several steps. In Lemma
\ref{L:Convexity1} we prove that $L_{2}^{o}$ is convex in $\beta$ and finite.
One of the consequences of this lemma is that the subdifferential of
$L_{2}^{o}(x,\cdot)$ is non empty. This result is used by Lemma
\ref{L:Convexity2} where we rewrite $L_{2}^{o}$ in the spirit of a Lagrange
multiplier problem where the role of the Lagrange multiplier is played by an
element in the subdifferential of $L_{2}^{o}(x,\cdot)$. Then, using Lemma
\ref{L:Convexity2} we prove in Lemma \ref{L:BoundedOptimalControlRegime2} that
an optimal control exists which is bounded and Lipschitz continuous in $y$.
Lemma \ref{L:DifferentiabilityMainLemma} uses Lemmas \ref{L:Convexity1} and
\ref{L:BoundedOptimalControlRegime2} together with the technical Lemma
\ref{L:Differentiability1} to prove that the dual of $L_{2}^{o}(x,\beta)$ with
respect to $\beta$ is strictly convex, which implies that $L_{2}^{o}(x,\beta)$
is differentiable in $\beta$. In Lemma \ref{L:ContinuityLocalRateFunction} we
prove that $L_{2}^{o}(x,\beta)$ is continuous in $(x,\beta)\in\mathbb{R}^{d}$
using Lemmas \ref{L:Convexity1} and \ref{L:BoundedOptimalControlRegime2}.
Lastly, in Lemma \ref{L:ContinuousOptimalControlRegime2} we prove that the
control that is constructed in the proof of Lemma
\ref{L:BoundedOptimalControlRegime2} is continuous in $x$, which together with
uniqueness of the corresponding invariant measure imply that the latter is
weakly continuous in $x$. Theorem \ref{T:BoundedOptimalControlRegime2} follows
from Lemmas \ref{L:BoundedOptimalControlRegime2} and
\ref{L:ContinuousOptimalControlRegime2}. Theorem
\ref{T:LocalRateFunctionRegime2} follows from Lemmas \ref{L:Convexity1},
\ref{L:DifferentiabilityMainLemma} and \ref{L:ContinuityLocalRateFunction}.

For the reader's convenience we recall%
\begin{align*}
\mathcal{A}_{x,\beta}^{2,r}  &  =\left\{  \mathrm{P}\in\mathcal{P}%
(\mathcal{Z}\times\mathcal{Y}):\int_{\mathcal{Z}\times\mathcal{Y}}%
\mathcal{L}_{z,x}^{2}f(y)\mathrm{P}(dzdy)=0\text{ for all }f\in C^{2}%
(\mathcal{Y})\right. \\
&  \left.  \hspace{2cm}\int_{\mathcal{Z}\times\mathcal{Y}}\left\Vert
z\right\Vert ^{2}\mathrm{P}(dzdy)<\infty\text{ and }\beta=\int_{\mathcal{Z}%
\times\mathcal{Y}}\lambda_{2}(x,y,z)\mathrm{P}(dzdy)\right\}  .
\end{align*}

For notational convenience, we ignore for the moment the $x-$dependence since
this is seen as parameter by the local rate function. Sometimes, the analysis
works with the relaxed form of the local rate, but as noted previously
$L_{2}^{o}(\beta)=L_{2}^{r}(\beta)$.

\begin{lemma}
\label{L:Convexity1} The cost $L_{2}^{r}(\beta)$ is a finite and convex
function of $\beta$.
\end{lemma}

\noindent\textit{Proof.} Given $\beta$ let $v_{\beta}(y)=\sigma^{-1}%
(y)(\beta-\gamma b(y)-c(y))$. Then $v_{\beta}(y)$ is Lipschitz continuous, and
hence there is an associated unique invariant distribution $\mu_{\beta}(dy)$.
Letting $\mathrm{P}(dzdy)=\delta_{v_{\beta}(y)}(dz)\mu_{\beta}(dy)$, we have
\[
\int_{\mathcal{Z}\times\mathcal{Y}}(\gamma b(y)+c(y)+\sigma(y)z)\mathrm{P}%
(dzdy)=\int_{\mathcal{Y}}\beta\mu_{\beta}(dy)=\beta,
\]
and similarly the first condition for inclusion in $\mathcal{A}_{\beta}^{2,r}$
can be checked. Since $v_{\beta}(y)$ is bounded the associated cost is finite,
and so $L_{2}^{r}(\beta)<\infty$.

Next let $\beta_{1},\beta_{2}\in\mathbb{R}^{d}$ and denote by $\mathrm{P}%
_{1},\mathrm{P}_{2}$ corresponding controls such that $\int_{\mathcal{Z}%
\times\mathcal{Y}}\lambda(y,z)\mathrm{P}_{i}(dzdy)=\beta_{i}$. Consider a
parameter $\eta\in\lbrack0,1]$ and define $\mathrm{P}_{0}=\eta\mathrm{P}%
_{1}+(1-\eta)\mathrm{P}_{2}$. Due to the linearity of integration,
$\mathrm{P}_{0}\in A_{\eta\beta_{1}+(1-\eta)\beta_{2}}^{2,r}$, and therefore
\begin{align*}
L_{2}^{r}(\eta\beta_{1}+(1-\eta)\beta_{2})  &  \leq\int_{\mathcal{Z}%
\times\mathcal{Y}}\frac{1}{2}\left\Vert z\right\Vert ^{2}\mathrm{P}%
_{0}(dzdy)\\
&  =\eta\int_{\mathcal{Z}\times\mathcal{Y}}\frac{1}{2}\left\Vert z\right\Vert
^{2}\mathrm{P}_{1}(dzdy)+(1-\eta)\int_{\mathcal{Z}\times\mathcal{Y}}\frac
{1}{2}\left\Vert z\right\Vert ^{2}\mathrm{P}_{2}(dzdy).
\end{align*}
Taking the infimum over all admissible $\mathrm{P}_{1},\mathrm{P}_{2}$ we get
\[
L_{2}^{r}(\eta\beta_{1}+(1-\eta)\beta_{2})\leq\eta L_{2}^{r}(\beta
_{1})+(1-\eta)L_{2}^{r}(\beta_{2}).
\]
This proves the convexity, and completes the proof of the lemma.\hfill
$\square\medskip$

For any $\beta\in\mathbb{R}^{d}$ the subdifferential of $L_{2}^{r}$ at $\beta$
is defined by%
\[
\partial L_{2}^{r}(\beta)=\{\zeta\in\mathbb{R}^{d}:L_{2}^{r}(\beta^{\prime
})-L_{2}^{r}(\beta)\geq\zeta\cdot(\beta^{\prime}-\beta)\hspace{0.2cm}\text{for
all }\beta^{\prime}\in\mathbb{R}^{d}\}.
\]
Since $L_{2}^{r}$ is finite and convex $\partial L_{2}^{r}(\beta)$ is always
nonempty. Define
\begin{align}
\mathcal{B}^{2,r}  &  =\left\{  \mathrm{P}\in\mathcal{P}(\mathcal{Z}%
\times\mathcal{Y}):\int_{\mathcal{Z}\times\mathcal{Y}}\left\Vert z\right\Vert
^{2}\mathrm{P}(dzdy)<\infty, \int_{\mathcal{Z}\times\mathcal{Y}}%
\mathcal{L}_{z}^{2}f(y)\mathrm{P}(dzdy)=0\text{ for all }f\in C^{2}%
(\mathcal{Y})\right\} \nonumber\\
\mathcal{B}^{2,o}  &  =\left\{  v(\cdot):\mathcal{Y}\mapsto\mathbb{R}^{d}%
,\mu\in\mathcal{P}(\mathcal{Y}): \int_{\mathcal{Y}}\left\Vert v(y)\right\Vert
^{2}\mu(dy)<\infty, \int_{\mathcal{Y}}\mathcal{L}_{v(y)}^{2}f(y)\mu(dy)=0
\text{ for all }f\in C^{2}(\mathcal{Y})\right\} \nonumber
\end{align}
and for $\zeta\in\mathbb{R}^{d}$ let
\begin{equation*}
\tilde{L}_{2}^{r}(\zeta)=\inf_{\mathrm{P}\in\mathcal{B}^{2,r}}\int
_{\mathcal{Z}\times\mathcal{Y}}\left(  \frac{1}{2}\left\Vert z\right\Vert
^{2}-\zeta\cdot(\gamma b(y)+c(y)+\sigma(y)z)\right)  \mathrm{P}(dzdy).
\label{Def:ErgodicControlConstraints2_1}%
\end{equation*}

We have the following lemma.

\begin{lemma}
\label{L:Convexity2} Consider any $\beta\in\mathbb{R}^{d}$ and any
$\zeta_{\beta}\in\partial L^{r}_{2}(\beta)$. Then
\[
\tilde{L}^{r}_{2}(\zeta_{\beta})=L^{r}_{2}(\beta)-\zeta_{\beta}\cdot\beta.
\]

\end{lemma}

\noindent\textit{Proof.} First we prove that $\tilde{L}_{2}^{r}(\zeta_{\beta
})\leq L_{2}^{r}(\beta)-\zeta_{\beta}\cdot\beta$, which follows from%
\begin{align*}
L_{2}^{r}(\beta)-\zeta_{\beta}\cdot\beta &  =\inf_{\mathrm{P}\in
\mathcal{A}_{\beta}^{2,r}}\int_{\mathcal{Z}\times\mathcal{Y}}\frac{1}%
{2}\left\Vert z\right\Vert ^{2}\mathrm{P}(dzdy)-\zeta_{\beta}\cdot\beta\\
&  =\inf_{\mathrm{P}\in\mathcal{A}_{\beta}^{2,r}}\int_{\mathcal{Z}%
\times\mathcal{Y}}\left(  \frac{1}{2}\left\Vert z\right\Vert ^{2}-\zeta
_{\beta}\cdot(\gamma b(y)+c(y)+\sigma(y)z)\right)  \mathrm{P}(dzdy)\\
&  \geq\inf_{\mathrm{P}\in\mathcal{B}^{2,r}}\int_{\mathcal{Z}\times
\mathcal{Y}}\left(  \frac{1}{2}\left\Vert z\right\Vert ^{2}-\zeta_{\beta}%
\cdot(\gamma b(y)+c(y)+\sigma(y)z)\right)  \mathrm{P}(dzdy)\\
&  =\tilde{L}_{2}^{r}(\zeta_{\beta}).
\end{align*}

For the opposite direction we use that $\zeta_{\beta}\in\partial L_{2}%
^{r}(\beta)$. Consider any $\beta^{\prime}\in\mathbb{R}^{d}$ and any
$\mathrm{P}\in A_{\beta^{\prime}}^{2,r}$. Then%
\begin{align}
L_{2}^{r}(\beta)-\zeta_{\beta} \cdot\beta &  \leq L_{2}^{r}(\beta^{\prime
})-\zeta_{\beta}\cdot\beta^{\prime}\nonumber\\
&  \leq\int_{\mathcal{Z}\times\mathcal{Y}}\frac{1}{2}\left\Vert z\right\Vert
^{2}\mathrm{P}(dzdy)-\zeta_{\beta}\cdot\beta^{\prime}\nonumber\\
&  =\int_{\mathcal{Z}\times\mathcal{Y}}\left(  \frac{1}{2}\left\Vert
z\right\Vert ^{2}-\zeta_{\beta}\cdot(\gamma b(y)+c(y)+\sigma(y)z)\right)
\mathrm{P}(dzdy).\nonumber
\end{align}
Since $\mathcal{B}^{2,r}=\cup_{\beta^{\prime}\in\mathbb{R}^{d}}A_{\beta
^{\prime}}^{2,r}$, the last display implies%
\[
L_{2}^{r}(\beta)-\zeta_{\beta}\cdot\beta\leq\inf_{\mathrm{P}\in\mathcal{B}%
^{2,r}}\int_{\mathcal{Z}\times\mathcal{Y}}\left(  \frac{1}{2}\left\Vert
z\right\Vert ^{2}-\zeta_{\beta} \cdot(\gamma b(y)+c(y)+\sigma(y)z)\right)
\mathrm{P}(dzdy)=\tilde{L}_{2}^{r}(\zeta_{\beta}).
\]
This concludes the proof of the lemma.\hfill$\square$

\begin{lemma}
\label{L:BoundedOptimalControlRegime2} Assume Condition \ref{A:Assumption1}.
Then there is a pair $(\bar{u},\bar{\mu})$ that achieves the infimum in the
definition of the local rate function such that $\bar{u}=\bar{u}_{\beta}(x,y)$
is, for any fixed $(x,\beta)\in\mathbb{R}^{2d}$, bounded and Lipschitz
continuous in $y$. Also, $\bar{\mu}(dy)=\bar{\mu}_{\bar{u}}(dy|x)$ is the
unique invariant measure corresponding to the operator $\mathcal{L}^{2}%
_{\bar{u}_{\beta}(x,y),x}$.
\end{lemma}

\begin{proof}
By Lemma \ref{L:Convexity2} we get that for any $\beta\in\mathbb{R}^{d}$ and
$\zeta_{\beta}\in\partial L_{2}^{r}(\beta)$,%
\begin{equation*}
L_{2}^{r}(\beta)=\tilde{L}_{2}^{r}(\zeta_{\beta})+\zeta_{\beta}\cdot\beta
=\inf_{\mathrm{P}\in\mathcal{B}^{2,r}}\int_{\mathcal{Z}\times\mathcal{Y}%
}\left(  \frac{1}{2}\left\Vert z\right\Vert ^{2}-\zeta_{\beta}\cdot(\gamma
b(y)+c(y)+\sigma(y)z-\beta)\right)  \mathrm{P}(dzdy).
\label{Def:OptimizationProblem3}%
\end{equation*}
According to Theorem 6.1 in \cite{KurtzStockbridge}, this optimization also
has a representation via an ergodic control problem of the form%
\begin{equation*}
\tilde{L}_{2}^{r}(\zeta_{\beta})+\zeta_{\beta}\cdot\beta=\inf\limsup
_{T\rightarrow\infty}\frac{1}{T}\mathbb{E}\int_{0}^{T}\left(  \frac{1}%
{2}\left\Vert v_{s}\right\Vert ^{2}-\zeta_{\beta}\cdot(\gamma b(Y_{s}%
)+c(Y_{s})+\sigma(Y_{s})v_{s}-\beta)\right)  ds,
\label{Def:OptimizationProblem3_1}%
\end{equation*}
where the infimum is over all progressively measurable controls $v$ and
solutions to the controlled martingale problem associated with $\mathcal{L}%
_{z}^{2}$. [The paper \cite{KurtzStockbridge} works with relaxed controls, but
since here the dynamics are affine in the control and the cost is convex, the
infima over relaxed and ordinary controls are the same.]

The Bellman equation associated with this control problem is
\begin{equation}
\inf_{v}\left[  \mathcal{L}_{v}^{2}W(y)+\frac{1}{2}\left\Vert v\right\Vert
^{2}-\zeta_{\beta}\cdot(\gamma b(y)+c(y)+\sigma(y)v-\beta)\right]  =\rho.
\label{Eq:ErgodicControlProblemRegime2_1}%
\end{equation}
Using the standard vanishing discount approach and taking into account the
periodicity condition (see for example \cite{ArisawaLions,
ArapostathisBorkarGhosh}) one can show that there is a unique pair
$(W,\rho)\in\mathcal{C}^{2}(\mathbb{R}^{d})\times\mathbb{R}$, such that
$W(0)=0$ and $W(y)$ is periodic in $y$ with period $1$ that satisfies
(\ref{Eq:ErgodicControlProblemRegime2_1}). Since we have a classical sense
solution, by the verification theorem for ergodic control $\rho=\rho
(\beta)=L_{2}^{r}(\beta)=\tilde{L}_{2}^{r}(\zeta_{\beta})+\zeta_{\beta}%
\cdot\beta$. In order to emphasize the dependence of $W(\cdot)$ on $\beta$ we
write $W(y)=W_{\beta}(y)$. It also follows from the verification argument that
an optimal control $\bar{u}_{\beta}(y)$ is given by $\bar{u}_{\beta
}(y)=-\sigma(y)^{T}(\nabla_{y}W_{\beta}(y)-\zeta_{\beta})$. Compactness of the
state space and the assumptions on the coefficients guarantee that the
gradient of $W_{\beta}(\cdot)$ is bounded, i.e., $\left\Vert \nabla
_{y}W_{\beta}\right\Vert \leq K(\beta)$ for some constant $K(\beta)$ that may
depend on $\beta$. Therefore, such an optimal control is indeed bounded and
Lipschitz continuous in $y$. Existence and uniqueness of the invariant measure
follows from the latter and the non-degeneracy assumption.
\end{proof}

Next, we prove that the local rate function $L_{2}^{o}(\beta)$ is actually
differentiable in $\beta\in\mathbb{R}^{d}$. Recall the operator
\[
\mathcal{L}_{u(y)}^{2}=\left[  \gamma b(y)+c(y)+\sigma(y)u(y)\right]
\cdot\nabla_{y}+\gamma\frac{1}{2}\sigma(y)\sigma(y)^{T}:\nabla_{y}\nabla_{y}.
\]
For notational convenience we omit the superscript $2$ and write
$\mathcal{L}_{u(y)}$ in place of $\mathcal{L}_{u(y)}^{2}$. Recall also that
for a bounded and Lipschitz continuous control $\bar{u}$ there exists a unique
invariant measure $\mu(dy)$ corresponding to $\mathcal{L}_{\bar{u}(y)}$.

Define the set of functions
\[
\mathcal{H}\doteq\left\{  h:\mathcal{Y}\mapsto\mathbb{R}\text{ such that
}h\text{ is periodic, bounded, Lipschitz continuous and }\int_{\mathcal{Y}%
}h(y)\mu(dy)=1\right\}  .
\]
For a vector $\theta\in\mathbb{R}^{d}$, $\eta\in\mathbb{R}$ and $h\in
\mathcal{H}$ define the perturbed control
\begin{equation}
\bar{u}_{\eta}(y)\doteq\bar{u}(y)+\eta\sigma(y)^{-1}\theta h(y).
\label{Eq:DifferentiabilityAuxiliaryControl}%
\end{equation}
For each $\eta$ there is a unique invariant measure $\mu_{\eta}(dy)$
corresponding to $\mathcal{L}_{\eta}=\mathcal{L}_{\bar{u}_{\eta}(y)}$, and it
is straightforward to show that $\mu_{\eta}(dy)\rightarrow\mu(dy)$ in the weak
topology as $|\eta|\downarrow0$. Moreover, under Condition \ref{A:Assumption1}%
, Lemma 3.2 in \cite{BuckdahnHuPeng} guarantees that the invariant measures
$\mu_{\eta}(dy)$ and $\mu(dy)$ have densities $m_{\eta}(y)$ and $m(y)$
respectively. In particular, there exist unique weak sense solutions to the
equations
\begin{equation*}
\mathcal{L}_{\eta}^{\ast}m_{\eta}(y)=0,\hspace{0.1cm}\int_{\mathcal{Y}}%
m_{\eta}(y)dy=1\text{ and }\mathcal{L}_{\bar{u}}^{\ast}m(y)=0,\hspace
{0.1cm}\int_{\mathcal{Y}}m(y)dy=1
\end{equation*}
where $\mathcal{L}_{\eta}^{\ast}$ and $\mathcal{L}_{\bar{u}}^{\ast}$ are the
formal adjoint operators to $\mathcal{L}_{\eta}$ and $\mathcal{L}_{\bar{u}}$
respectively. The densities are strictly positive, continuous and in
$H^{1}(\mathcal{Y})$. Observe that
\begin{equation}
\mathcal{L}_{\eta}^{\ast}m_{\eta}(y)=0\Leftrightarrow\mathcal{L}_{\bar{u}%
}^{\ast}m_{\eta}(y)=\eta\theta\cdot\nabla\left(  h(y)m_{\eta}(y)\right)  .
\label{Eq:DifferentiabilityPerturbedOperator1}%
\end{equation}
in the weak sense.

Next, for $g\in L^{2}(\mathcal{Y})$ consider the auxiliary partial
differential equation
\begin{equation}
\mathcal{L}_{\bar{u}}f(y)=g(y)-\int_{\mathcal{Y}}g(y)\mu(dy),\hspace
{0.1cm}f\text{ is 1 periodic and }\int_{\mathcal{Y}}f(y)\mu(dy)=0.
\label{Eq:DifferentiabilityAuxiliaryPDE}%
\end{equation}
By the Fredholm alternative and the strong maximum principle this equation has
a unique solution. Standard elliptic regularity theory yields $f\in
H^{2}(\mathbb{R}^{d})$. Then by Sobolev's embedding lemma we have that $f\in
C^{1}(\mathbb{R}^{d})$.

Denote by $(\cdot,\cdot)_{2}$ the usual inner product in $L^{2}(\mathcal{Y})$.
The following lemma will be useful in the sequel.

\begin{lemma}
\label{L:Differentiability1} Let $g\in L^{2}(\mathcal{Y})$, $\eta\in
\mathbb{R}$, $h\in\mathcal{H}$ and $f\in H^{2}(\mathbb{R}^{d})$ the solution
to (\ref{Eq:DifferentiabilityAuxiliaryPDE}). Then,
\begin{equation*}
\left(  g, (m_{\eta}-m)\right)  _{2}= -\eta\left(  \theta\cdot\nabla f,
hm\right)  _{2}-\eta\left(  \theta\cdot\nabla f, h( m_{\eta}-m)\right)  _{2}%
\end{equation*}

\end{lemma}

\begin{proof}
Keeping in mind (\ref{Eq:DifferentiabilityPerturbedOperator1}) and that
$m_{\eta}(y)$ and $m(y)$ are densities, the following hold
\begin{align*}
\left(  f,\mathcal{L}_{\bar{u}}^{\ast}(m_{\eta}-m)\right)  _{2}  &
=\eta\left(  f,\theta\cdot\nabla(hm_{\eta})\right)  _{2}\Rightarrow\nonumber\\
\left(  \mathcal{L}_{\bar{u}}f,(m_{\eta}-m)\right)  _{2}  &  =-\eta\left(
\theta\cdot\nabla f,hm_{\eta}\right)  _{2}\Rightarrow\nonumber\\
\left(  g,(m_{\eta}-m)\right)  _{2}  &  =-\eta\left(  \theta\cdot\nabla
f,hm\right)  _{2}-\eta\left(  \theta\cdot\nabla f,h(m_{\eta}-m)\right)  _{2}.
\end{align*}
This concludes the proof of the lemma.
\end{proof}

By Lemma \ref{L:Convexity1} we already know that $L_{2}^{o}$ is finite and
convex. To show that $L_{2}^{o}$ is differentiable, it is enough to show its
Legendre transform is strictly convex. For $\alpha\in\mathbb{R}^{d}$ define
\begin{align}
H(\alpha)  &  \doteq\sup_{\beta\in\mathbb{R}^{d}}\left[  \left\langle
\alpha,\beta\right\rangle -L_{2}^{o}(\beta)\right] \nonumber\\
&  =\sup_{(v,\mu)\in\mathcal{B}^{2,o}}\left[  \left\langle \alpha
,\int_{\mathcal{Y}}\left(  \gamma b(y)+c(y)+\sigma(y)v(y)\right)
\mu(dy)\right\rangle -\int_{\mathcal{Y}}\frac{1}{2}\left\Vert v(y)\right\Vert
^{2}\mu(dy)\right]  . \label{Eq:DualFunction}%
\end{align}

\begin{lemma}
\label{L:DifferentiabilityMainLemma} The Legrendre transform $H$ of $L_{2}%
^{o}$ is a strictly convex function of $\alpha\in\mathbb{R}^{d}$.
\end{lemma}

\begin{proof}
Suppose that $H$ is not strictly convex. Then there are $\alpha_{i}%
\in\mathbb{R}^{d},i=1,2$ not equal such that for all $\xi\in\lbrack0,1]$
\begin{align*}
H(\xi\alpha_{1}+(1-\xi)\alpha_{2})  &  =\xi H(\alpha_{1})+(1-\xi)H(\alpha
_{2})\\
&  =\left\langle \xi\alpha_{1}+(1-\xi)\alpha_{2},\int_{\mathcal{Y}}(\gamma
b(y)+c(y)+\sigma(y)\bar{u}(y))\mu(dy)\right\rangle -\int_{\mathcal{Y}}\frac
{1}{2}\left\Vert \bar{u}\right\Vert ^{2}\mu(dy),
\end{align*}
where $\bar{\beta}\doteq\int_{\mathcal{Y}}(\gamma b(y)+c(y)+\sigma(y)\bar
{u}(y))\mu(dy)\in\partial H(\xi\alpha_{1}+(1-\xi)\alpha_{2})$ for all $\xi
\in\lbrack0,1]$. As in Lemma \ref{L:BoundedOptimalControlRegime2}, it can be
shown that $\bar{u}$ exists and can be chosen to be bounded and Lipschitz
continuous. Also, $\mu$ is the unique invariant measure corresponding to the
operator $\mathcal{L}_{\bar{u}(y)}$. We will argue that the last display is impossible.

First observe that by subtracting $\left\langle \alpha,\bar{\beta
}\right\rangle $ we can arrange that $H$ is constant for $\alpha=\xi\alpha
_{1}+(1-\xi)\alpha_{2}$, $\xi\in\lbrack0,1]$. Let
\[
\bar{H}(\alpha)=H(\alpha)-\left\langle \alpha,\bar{\beta}\right\rangle
=-\int_{\mathcal{Y}}\frac{1}{2}\left\Vert \bar{u}\right\Vert ^{2}\mu(dy).
\]
Consider $\xi=1/2+\eta$ with $\eta$ small (and possibly negative). We will
construct $(v,\mu)\in\mathcal{B}^{2,o}$ that will give a lower bound for
$H(\xi\alpha_{1}+(1-\xi)\alpha_{2})$ through (\ref{Eq:DualFunction}) that is
strictly bigger than $-\int_{\mathcal{Y}}\frac{1}{2}\left\Vert \bar
{u}\right\Vert ^{2}\mu(dy)$. This contradicts the constancy of $H(\xi
\alpha_{1}+(1-\xi)\alpha_{2})$ for $\xi\in\lbrack0,1]$, and thus implies that
$H$ is strictly convex.

Define $\bar{u}_{\eta}(y)$ by (\ref{Eq:DifferentiabilityAuxiliaryControl})
with $\theta\doteq\alpha_{1}-\alpha_{2}$ and $h\in\mathcal{H}$. For
$\xi=1/2+\eta$ we have $\alpha=\frac{1}{2}(\alpha_{1}+\alpha_{2})+\eta
(\alpha_{1}-\alpha_{2})$. The definition of $H(\alpha)$ by
(\ref{Eq:DualFunction}) implies%
\begin{align*}
\bar{H}(\alpha)  &  \geq\left\langle \alpha,\eta(\alpha_{1}-\alpha_{2}%
)\int_{\mathcal{Y}}h(y)m_{\eta}(y)dy\right\rangle +\int_{\mathcal{Y}%
}\left\langle \alpha,\left(  \gamma b(y)+c(y)+\sigma(y)\bar{u}(y)\right)
\right\rangle (m_{\eta}(y)-m(y))dy\\
&  \mbox{}-\int_{\mathcal{Y}}\frac{1}{2}\left\Vert \bar{u}(y)+\eta\sigma
^{-1}(y)(\alpha_{1}-\alpha_{2})h(y)\right\Vert ^{2}m_{\eta}(y)dy.
\end{align*}
For $i=1,\ldots,d$, let $\phi_{i}(y)$ be the solution to
(\ref{Eq:DifferentiabilityAuxiliaryPDE}) with $g(y)=q_{i}(y)$, the
$i^{\mathrm{th}}$ component of $q(y)=\gamma b(y)+c(y)+\sigma(y)\bar{u}(y)$. We
write $\phi=(\phi_{1},\ldots,\phi_{d})$, and also denote by $\psi(y)$ the
solution to (\ref{Eq:DifferentiabilityAuxiliaryPDE}) with $g(y)=\left\Vert
\bar{u}(y)\right\Vert ^{2}$. Then by Lemma \ref{L:Differentiability1} the last
display can be rewritten as%
\begin{align*}
\bar{H}(\alpha)  &  \geq\frac{1}{2}\eta\left[  \left\langle \alpha_{1}%
+\alpha_{2},(\alpha_{1}-\alpha_{2})\right\rangle -\int_{\mathcal{Y}%
}\left\langle \alpha_{1}+\alpha_{2},\frac{\partial\phi}{\partial y}%
(y)(\alpha_{1}-\alpha_{2})\right\rangle h(y)m(y)dy\right. \\
&  \mbox{}-2\int_{\mathcal{Y}}\left\langle \bar{u}(y),\sigma^{-1}%
(y)(\alpha_{1}-\alpha_{2})\right\rangle h(y)m(y)dy\\
&  \mbox{}\left.  +\int_{\mathcal{Y}}\left\langle \alpha_{1}-\alpha_{2}%
,\nabla\psi(y)\right\rangle h(y)m(y)dy\right]  -\frac{1}{2}\int_{\mathcal{Y}%
}\left\Vert \bar{u}(y)\right\Vert ^{2}m(y)dy+o(\eta)
\end{align*}
where $o(\eta)$ is such that $o(\eta)/\eta\downarrow0$ as $|\eta|\downarrow0$
and can be neglected.

Now for small $\eta$ (perhaps negative) this is strictly bigger than
$-\frac{1}{2}\int_{\mathcal{Y}}\left\Vert \bar{u}\right\Vert ^{2}\mu(dy)$
unless the $O(\eta)$ term is zero, i.e., unless%
\begin{align*}
0  &  =\left\langle \alpha_{1}+\alpha_{2},(\alpha_{1}-\alpha_{2})\right\rangle
-\int_{\mathcal{Y}}\left\langle \alpha_{1}+\alpha_{2},\frac{\partial\phi
}{\partial y}(y)(\alpha_{1}-\alpha_{2})\right\rangle h(y)m(y)dy\\
&  \mbox{}-2\int_{\mathcal{Y}}\left\langle \bar{u}(y),\sigma^{-1}%
(y)(\alpha_{1}-\alpha_{2})\right\rangle h(y)m(y)dy\\
&  \mbox{}+\int_{\mathcal{Y}}\left\langle \alpha_{1}-\alpha_{2},\nabla
\psi(y)\right\rangle h(y)m(y)dy.
\end{align*}
However, in the argument by contradiction $\alpha_{1}$ and $\alpha_{2}$ can be
replaced by any $\varepsilon_{1}\alpha_{1}+(1-\varepsilon_{1})\alpha_{2}$ and
$\varepsilon_{2}\alpha_{1}+(1-\varepsilon_{2})\alpha_{2}$, so long as
$0\leq\varepsilon_{1}<\varepsilon_{2}\leq1$. After performing this
substitution and some algebra, the last display becomes
\begin{align*}
0  &  =(\epsilon_{1}^{2}-\epsilon_{2}^{2})\int_{\mathcal{Y}}\left\langle
\alpha_{1}-\alpha_{2},\left(  I-\frac{\partial\phi}{\partial y}(y)\right)
(\alpha_{1}-\alpha_{2})\right\rangle h(y)m(y)dy\\
&  \mbox{}+2(\epsilon_{1}-\epsilon_{2})\left[  \int_{\mathcal{Y}}\left\langle
\alpha_{2},\left(  I-\frac{\partial\phi}{\partial y}(y)\right)  (\alpha
_{1}-\alpha_{2})\right\rangle h(y)m(y)dy\right. \\
&  \mbox{}\left.  -\int_{\mathcal{Y}}\left\langle \bar{u}(y),\sigma
^{-1}(y)(\alpha_{1}-\alpha_{2})\right\rangle h(y)m(y)dy+\frac{1}{2}%
\int_{\mathcal{Y}}\left\langle \alpha_{1}-\alpha_{2},\nabla\psi
(y)\right\rangle h(y)m(y)dy\right]  .
\end{align*}
We claim that the last display cannot be true since $\epsilon_{1}\neq
\epsilon_{2}$ and $\alpha_{1}-\alpha_{2}\neq0$. By considering various choices
for $\epsilon_{1}$ and $\epsilon_{2}$, it is enough show that the term
multiplying $(\epsilon_{1}^{2}-\epsilon_{2}^{2})$ is not zero for all
$h\in\mathcal{H}$. Let us assume the contrary, and that for all $h\in
\mathcal{H}$
\begin{equation}
\int_{\mathcal{Y}}\left\langle \alpha_{1}-\alpha_{2},\left(  I-\frac
{\partial\phi}{\partial y}(y)\right)  (\alpha_{1}-\alpha_{2})\right\rangle
h(y)m(y)dy=0. \label{Eq:FalseStetement1}%
\end{equation}
This implies that
\begin{equation}
\left\langle \alpha_{1}-\alpha_{2},\frac{\partial\phi}{\partial y}%
(y)(\alpha_{1}-\alpha_{2})\right\rangle =\left\Vert \alpha_{1}-\alpha
_{2}\right\Vert ^{2}\text{ for all }y\in\mathcal{Y}.
\label{Eq:FalseStetement3}%
\end{equation}
Define
\[
\Phi(y)\doteq(\alpha_{1}-\alpha_{2})\cdot\phi(y)
\]
Then $\Phi$ is a periodic, bounded and $C^{1}(\mathbb{R}^{d})$ function.
Consider any trajectory $\zeta_{t}:\mathbb{R}_{+}\mapsto\mathcal{Y}$ such that
$\dot{\zeta}_{t}=(\alpha_{1}-\alpha_{2})$. Differentiation of $\Phi(\zeta
_{t})$ and use of (\ref{Eq:FalseStetement3}) give
\[
\frac{d}{dt}\Phi(\zeta_{t})=\left\langle \alpha_{1}-\alpha_{2},\frac
{\partial\phi}{\partial y}(\zeta_{t})(\alpha_{1}-\alpha_{2})\right\rangle
=\left\Vert \alpha_{1}-\alpha_{2}\right\Vert ^{2}>0,
\]
which cannot be true due to the periodicity and boundedness of $\Phi$. This
implies that (\ref{Eq:FalseStetement1}) is false, i.e., that there is
$h\in\mathcal{H}$ such that
\[
\int_{\mathcal{Y}}\left\langle \alpha_{1}-\alpha_{2},\left(  I-\frac
{\partial\phi}{\partial y}(y)\right)  (\alpha_{1}-\alpha_{2})\right\rangle
h(y)m(y)dy\neq0.
\]
This concludes the proof of the lemma.
\end{proof}

Let us now recall the $x-$dependence and prove that the local rate function
$L_{2}^{o}(x,\beta)$ is continuous in $(x,\beta)\in\mathbb{R}^{2d}$.

\begin{lemma}
\label{L:ContinuityLocalRateFunction} The local rate function $L_{2}%
^{o}(x,\beta)$ is continuous in $(x,\beta)\in\mathbb{R}^{2d}$.
\end{lemma}

\begin{proof}
First, we prove that $L_{2}^{o}(x,\beta)$ is lower semicontinuous in
$(x,\beta)\in\mathbb{R}^{2d}$. We work with the relaxed formulation of the
local rate function, but as noted previously $L_{2}^{r}(x,\beta)=L_{2}%
^{o}(x,\beta)$.

Consider $\{(x_{n},\beta_{n})\in\mathbb{R}^{2d}\}_{n\in\mathbb{N}}$ such that
$(x_{n},\beta_{n})\rightarrow(x,\beta)$. We want to prove
\[
\liminf_{n\rightarrow\infty}L_{2}^{r}(x_{n},\beta_{n})\geq L_{2}^{r}%
(x,\beta).
\]
Let $M<\infty$ such that $\liminf_{n\rightarrow\infty}L_{2}^{r}(x_{n}%
,\beta_{n})\leq M$. The definition of $L_{2}^{r}(x_{n},\beta_{n})$ implies
that we can find measures $\{\mathrm{P}^{n},n<\infty\}$ satisfying
$\mathrm{P}^{n}\in\mathcal{A}_{x_{n},\beta_{n}}^{2,r}$ such that
\begin{equation}
\sup_{n<\infty}\frac{1}{2}\int_{\mathcal{Z}\times\mathcal{Y}}\left\Vert
z\right\Vert ^{2}\mathrm{P}^{n}(dzdy)<M+1 \label{L2bound}%
\end{equation}
and
\[
L_{2}^{r}(x_{n},\beta_{n})\geq\left[  \frac{1}{2}\int_{\mathcal{Z}%
\times\mathcal{Y}}\left\Vert z\right\Vert ^{2}\mathrm{P}^{n}(dzdy)-\frac{1}%
{n}\right]
\]
It follows from (\ref{L2bound}) and the definition of $\mathcal{A}_{x,\beta
}^{2,r}$ that $\{\mathrm{P}^{n},n<\infty\}$ is tight and any limit point
$\mathrm{P}$ of $\mathrm{P}^{n}$ will be in $\mathcal{A}_{x,\beta}^{2,r}$.
Hence by Fatou's Lemma
\begin{align*}
\liminf_{n\rightarrow\infty}L_{2}^{r}(x_{n},\beta_{n})  &  \geq\liminf
_{n\rightarrow\infty}\left[  \frac{1}{2}\int_{\mathcal{Z}\times\mathcal{Y}%
}\left\Vert z\right\Vert ^{2}\mathrm{P}^{n}(dzdy)-\frac{1}{n}\right] \\
&  \geq\frac{1}{2}\int_{\mathcal{Z}\times\mathcal{Y}\times}\left\Vert
z\right\Vert ^{2}\mathrm{P}(dzdy)\\
&  \geq\inf_{\mathrm{P}\in\mathcal{A}_{x,\beta}^{2,r}}\left[  \frac{1}{2}%
\int_{\mathcal{Z}\times\mathcal{Y}\times}\left\Vert z\right\Vert
^{2}\mathrm{P}(dzdy)\right] \\
&  =L_{2}^{r}(x,\beta),
\end{align*}
which concludes the proof of lower semicontinuity of $L_{2}^{r}(x,\beta
)=L_{2}^{o}(x,\beta)$.

Next we prove that $L_{2}^{o}(x,\beta)$ is upper semicontinuous. Fix
$(x,\beta)\in\mathbb{R}^{2d}$. By Lemma \ref{L:BoundedOptimalControlRegime2},
we know that the optimal control $\bar{u}=\bar{u}_{\beta}(x,y)$ exists and can
be chosen to be bounded and continuous in $y$. Hence, there is a unique
invariant measure corresponding to the operator $\mathcal{L}_{\bar{u}_{\beta
}(x,y),x}^{2}$ which will be denoted by $\bar{\mu}_{\bar{u}}(dy|x)$.

Let $\{x_{n}\in\mathbb{R}^{d}\}$ be such that $x_{n}\rightarrow x$ and define
a control $u^{n}$ by the formula
\begin{equation}
\gamma b(x_{n},y)+c(x_{n},y)+\sigma(x_{n},y)u^{n}(y)=\gamma
b(x,y)+c(x,y)+\sigma(x,y)\bar{u}_{\beta}(x,y).
\label{Eq:DefinitionOfControlContinuityOfL}%
\end{equation}
Since $\sigma(x,y)$ is nondegenerate, $u^{n}(y)$ is uniquely defined,
continuous in $y$ and uniformly bounded in $(n,y)$, i.e., there exists a
constant $M<\infty$ such that $\sup_{(n,y)\in\mathbb{N}\times\mathcal{Y}%
}\left\Vert u^{n}(y)\right\Vert \leq M$. It follows from
\[
\sigma(x_{n},y)\left[  u^{n}(y)-\bar{u}_{\beta}(x,y)\right]  =\left[
\sigma(x,y)-\sigma(x_{n},y)\right]  \bar{u}_{\beta}(x,y)+\gamma\left[
b(x,y)-b(x_{n},y)\right]  +\left[  c(x,y)-c(x_{n},y)\right]
\]
that in fact $u^{n}(y)$ converges to $\bar{u}_{\beta}(x,y)$ uniformly in $y$.
Since $u^{n}(y)$ is bounded and Lipschitz continuous there is a unique
invariant measure corresponding to $\mathcal{L}_{u^{n}(y),x_{n}}^{2}$ which
will be denoted by $\theta^{n}(dy)$.

Owing to the definition of $u^{n}(y)$ via
(\ref{Eq:DefinitionOfControlContinuityOfL}), the operator $\mathcal{L}%
_{u^{n}(y),x_{n}}^{2}$ takes the form
\[
\mathcal{L}_{u^{n}(y),x_{n}}^{2}=\left[  \gamma b(x,y)+c(x,y)+\sigma
(x,y)\bar{u}_{\beta}(x,y)\right]  \cdot\nabla_{y}+\gamma\frac{1}{2}%
\sigma(x_{n},y)\sigma(x_{n},y)^{T}:\nabla_{y}\nabla_{y}.
\]
Hence by Condition \ref{A:Assumption1}, it follows that $\theta^{n}%
(dy)\rightarrow\bar{\mu}_{\bar{u}}(dy|x)$ in the topology of weak convergence.
Let $\{\beta_{n}\in\mathbb{R}^{d}\}$ be defined by%
\begin{align}
\beta_{n}  &  =\int_{y\in\mathcal{Y}}\left(  \gamma b(x_{n},y)+c(x_{n}%
,y)+\sigma(x_{n},y)u^{n}(y)\right)  \theta^{n}(dy)\nonumber\\
&  =\int_{y\in\mathcal{Y}}\left(  \gamma b(x,y)+c(x,y)+\sigma(x,y)\bar
{u}_{\beta}(x,y)\right)  \theta^{n}(dy).\nonumber
\end{align}
Then the weak convergence $\theta^{n}(dy)\Rightarrow\bar{\mu}_{\bar{u}}%
(dy|x)$, the uniform convergence of $u^{n}(y)$ to $\bar{u}_{\beta}(x,y)$, and
the continuity in $y$ of the function $\gamma b(x,y)+c(x,y)+\sigma(x,y)\bar
{u}_{\beta}(x,y)$ imply that $\beta_{n}\rightarrow\beta$. Thus
\begin{align*}
\limsup_{n\rightarrow\infty}L_{2}^{o}(x_{n},\beta_{n})  &  =\limsup
_{n\rightarrow\infty}\inf_{(v,\mu)\in\mathcal{A}_{x_{n},\beta_{n}}^{2,r}%
}\left\{  \frac{1}{2}\int_{\mathcal{Y}}\left\Vert v(y)\right\Vert ^{2}%
\mu(dy)\right\} \\
&  \leq\limsup_{n\rightarrow\infty}\frac{1}{2}\int_{\mathcal{Y}}\left\Vert
u^{n}(y)\right\Vert ^{2}\theta^{n}(dy)\\
&  =\frac{1}{2}\int_{\mathcal{Y}}\left\Vert \bar{u}_{\beta}(x,y)\right\Vert
^{2}\bar{\mu}_{\bar{u}}(dy|x)\\
&  =L_{2}^{o}(x,\beta).
\end{align*}
Line $2$ follows from the choice of a particular control. Line $3$ follows
from the uniform convergence of $\left\Vert u^{n}(y)\right\Vert ^{2}$ to
$\left\Vert \bar{u}_{\beta}(x,y)\right\Vert ^{2}$, the continuity and
boundedness of $\bar{u}_{\beta}(x,y)$ in $y$, and the weak convergence
$\theta^{n}(dy)\Rightarrow\bar{\mu}_{\bar{u}}(dy|x)$. Line $4$ follows from
the fact that $\bar{u}$ is the control that achieves the infimum in the
definition of $L_{2}^{o}(x,\beta)$.

We have shown that if $x_{n}\rightarrow x$ then there exists $\{\beta_{n}%
\in\mathbb{R}^{d}\}$ such that $\beta_{n}\rightarrow\beta$ and $\limsup
_{n\rightarrow\infty}L_{2}^{o}(x_{n},\beta_{n})\leq L_{2}^{o}(x,\beta)$. We
claim that in fact the same is true for any sequence $\beta_{n}\rightarrow
\beta$. Let $\delta>0$ be given. Since $L_{2}^{o}(x,\cdot)$ is finite and
convex, we can choose $\rho^{j}>0,\gamma^{j}\in\mathbb{R}^{d},j=1,\ldots,d$,
such that the convex hull of $\gamma^{j}\in\mathbb{R}^{d},j=1,\ldots,d$ has
nonempty interior, $\sum_{j=1}^{d}\rho^{j}=1,\beta=\sum_{j=1}^{d}\rho
^{j}\gamma^{j},$ and
\[
L_{2}^{o}(x,\beta)\geq\sum_{j=1}^{d}\rho^{j}L_{2}^{o}(x,\gamma^{j})-\delta.
\]
For each $\gamma^{j}$ construct a sequence $\gamma_{n}^{j}$ such that
$\gamma_{n}^{j}\rightarrow\gamma^{j}$ and $\limsup_{n\rightarrow\infty}%
L_{2}^{o}(x_{n},\gamma_{n}^{j})\leq L_{2}^{o}(x,\gamma^{j})$. Since for all
sufficiently large $n$ $\beta_{n}$ is in the interior of the convex hull of
$\gamma^{j},j=1,\ldots,d$, there are for all such $n$ $\rho_{n}^{j}>0$ such
that $\sum_{j=1}^{d}\rho_{n}^{j}=1,$ $\beta_{n}=\sum_{j=1}^{d}\rho_{n}%
^{j}\gamma_{n}^{j},$ and $\rho_{n}^{j}\rightarrow\rho^{j}$. By convexity
\[
\limsup_{n\rightarrow\infty}L_{2}^{o}(x_{n},\beta_{n})\leq\limsup
_{n\rightarrow\infty}\sum_{j=1}^{d}\rho_{n}^{j}L_{2}^{o}(x_{n},\gamma_{n}%
^{j})\leq\sum_{j=1}^{d}\rho^{j}L_{2}^{o}(x,\gamma^{j})\leq L_{2}^{o}%
(x,\beta)+\delta.
\]
Letting $\delta\downarrow0$ concludes the proof of the lemma.
\end{proof}

\begin{lemma}
\label{L:ContinuousOptimalControlRegime2} The control $\bar{u}=\bar{u}_{\beta
}(x,y)$ constructed in the proof of Lemma \ref{L:BoundedOptimalControlRegime2}
is continuous in $x$, Lipschitz continuous in $y$ and measurable in
$(x,y,\beta) $. Moreover, the invariant measure $\bar{\mu}_{\bar{u}}(dy|x)$
corresponding to the operator $\mathcal{L}^{2}_{\bar{u}_{\beta}(x,y),x}$ is
weakly continuous as a function of $x$.
\end{lemma}

\begin{proof}
Recall that
\[
\bar{u}_{\beta}(x,y)=-\sigma(x,y)^{T}(\nabla_{y}W_{\beta}(x,y)-\zeta_{\beta
}(x)),
\]
where $\zeta_{\beta}(x)$ is a subdifferential of $L_{2}^{o}(x,\beta)$ at
$\beta$. By Lemma \ref{L:DifferentiabilityMainLemma}, the subdifferential of
$L_{2}^{o}(x,\beta)$ with respect to $\beta$ consists only of the gradient
$\nabla_{\beta}L_{2}^{o}(x,\beta)$. Then continuity of $\zeta_{\beta}(x) $
follows from this uniqueness and the joint continuity of $L_{2}^{o}(x,\beta)$
established in Lemma \ref{L:ContinuityLocalRateFunction}.

Lipschitz continuity in $y$ of $\bar{u}_{\beta}(x,y)$ was established in Lemma
\ref{L:BoundedOptimalControlRegime2}. We insert $\bar{u}_{\beta}(x,y)$ as the
optimizer into (\ref{Eq:ErgodicControlProblemRegime2_1}). Recall that
$\mathcal{L}_{z,x}^{2}$ is an operator in $y$ only and denote by
$\mathcal{L}_{0,x}^{2}$ the operator $\mathcal{L}_{z,x}^{2}$ with the control
variable $z=0$. After some rearrangement of terms we get the equation
\begin{equation*}
\mathcal{L}_{0,x}^{2}\bar{W}_{\beta}(x,y)-\frac{1}{2}\left\Vert \sigma
^{T}(x,y)\nabla_{y}\bar{W}_{\beta}(x,y)\right\Vert ^{2}=\bar{H}_{\beta}(x),
\label{Eq:ErgodicControlProblemRegime2_2}%
\end{equation*}
where $\nabla_{y}\bar{W}_{\beta}(x,y)=\nabla_{y}W_{\beta}(x,y)-\zeta_{\beta
}(x)$ and $\bar{H}_{\beta}(x)=\rho(x,\beta)-\zeta_{\beta}(x)\cdot\beta
=\tilde{L}_{2}^{r}(x,\zeta_{\beta})$. This is now in the standard form for the
Bellman equation of an ergodic control problem. As before a classical sense
solution exists, and as a consequence we have the representation
\[
\bar{H}_{\beta}(x)=\inf_{v}\limsup_{T\rightarrow\infty}\frac{1}{T}%
\mathbb{E}\int_{0}^{T}\left(  \frac{1}{2}\left\Vert v_{s}\right\Vert
^{2}-\zeta_{\beta}(x)\cdot(\gamma b(x,Y_{s})+c(x,Y_{s})+\sigma(x,Y_{s}%
)v_{s})\right)  ds,
\]
where the infimum is over all progressively measurable controls. Since by
Condition \ref{A:Assumption1} $b,c,$ and $\sigma$ are continuous in $x$
uniformly in $y$ and since $\zeta_{\beta}(x)$ is continuous in $x$, $\bar
{H}_{\beta}(x)$ is continuous in $x$.

A straight forward calculation shows that for any $x_{1},x_{2}\in\mathbb{R}^{d}$, the function $\Phi(y)=\bar{W}_{\beta}(x_{1},y)-\bar{W}_{\beta}(x_{2},y)$ satisfies a linear equation. This observation and the general theory for uniformly elliptic equations (see
\cite{GilbargTrudinger}) together with the continuity in $x$ of $\bar
{H}_{\beta}(x)$ and $\zeta_{\beta}(x)$ and Condition \ref{A:Assumption1} imply
that $\nabla_{y}\bar{W}_{\beta}(x,y)$ is continuous in $x$ as well. Hence, due
to the continuity of $\sigma$ we conclude that $\bar{u}_{\beta}(x,y)=-\sigma
(x,y)^{T}\nabla_{y}\bar{W}_{\beta}(x,y)$ is continuous in $x $. Measurability
is clear.

Lastly, due to continuity of the optimal control $\bar{u}$ in $x$, Condition
\ref{A:Assumption1} and uniqueness of $\bar{\mu}_{\bar{u}(x,y)}(dy|x)$ for
each $x$, we conclude that $\bar{\mu}_{\bar{u}(x,y)}(dy|x)$ is weakly
continuous as a function of $x$ (see, e.g., Section $3$ in
\cite{ArtsteinVigodner}).
\end{proof}

\section{Laplace principle upper bound for Regime 3}

\label{S:LaplacePrincipleRegime3}

In this section we discuss the Laplace principle upper bound for Regime 3. For
notational convenience we drop the superscript $3$ from $\bar{X}^{3}$ and
$\mathrm{P}^{3}$.

We consider the general multidimensional case when $c(x,y)=c(y)$ and
$\sigma(x,y)=\sigma(y)$. In Remark \ref{R:MultivaluedMapRegime3_2} we discuss
the case when the functions $c$ and $\sigma$ depend on $x$ as well. In
Subsection \ref{SS:Regime3alternativeFormula} we consider the $d=1$ case. For
$d=1$ we can establish the LDP when the coefficients depend on $x$ as well and
we provide an alternative expression for the rate function together with a
control that nearly achieves the large deviations lower bound at the prelimit
level. An easy computation shows that this alternate expression is equivalent
to the corresponding expression in \cite{FS} for $b(x,y)=b(y)$, $c(x,y)=c(y)$
and $\sigma(x,y)=\sigma(y)$ for $d=1$.

\begin{proof}
[Remarks on the proof of Laplace principle upper bound for Regime 3]For each
$\epsilon>0$, let $X^{\epsilon}$ be the unique strong solution to
(\ref{Eq:LDPandA1}). To prove the Laplace principle upper bound we must show
that for all bounded, continuous functions $h$ mapping $\mathcal{C}%
([0,1];\mathbb{R}^{d})$ into $\mathbb{R}$
\begin{equation*}
\limsup_{\epsilon\downarrow0}-\epsilon\ln\mathbb{E}_{x_{0}}\left[
\exp\left\{  -\frac{h(X^{\epsilon})}{\epsilon}\right\}  \right]  \leq
\inf_{(\phi,\mathrm{P})\in\mathcal{V}}\left[  \frac{1}{2}\int_{\mathcal{Z}%
\times\mathcal{Y}\times\lbrack0,1]}\left\Vert z\right\Vert ^{2}\mathrm{P}%
(dzdydt)+h(\phi)\right]  . \label{Eq:LaplacePrincipleUpperBound}%
\end{equation*}
Define
\[
I(\mathrm{P})=\frac{1}{2}\int_{\mathcal{Z}\times\mathcal{Y}\times\lbrack0,1]}
\left\Vert z\right\Vert ^{2}\mathrm{P}(dzdydt).
\]
Let $\eta>0$ be given and consider $(\psi,\bar{\mathrm{P}})\in\mathcal{V}$
with $\psi_{0}=x_{0}$ such that
\[
I(\bar{\mathrm{P}})+h(\psi)\leq\inf_{(\phi,\mathrm{P})\in\mathcal{V}}\left[
I(\mathrm{P})+h(\phi)\right]  +\eta<\infty.
\]
We claim that there is a family of controls $\{\bar{u}^{\epsilon}%
,\epsilon>0\}$ such that
\[
(\bar{X}^{\epsilon},\bar{\mathrm{P}}^{\epsilon,\Delta})\overset{\mathcal{D}%
}{\rightarrow}(\bar{X},\bar{\mathrm{P}}) \text{ and } \bar{X}=\psi\text{
w.p.}1,
\]
where $(\bar{X}^{\epsilon},\bar{\mathrm{P}}^{\epsilon,\Delta})$ is constructed
using $\bar{u}^{\epsilon}$. With this at hand the result easily follows.

The claim follows from the results in Section 3 in \cite{Gaitsgory}
and Section 4 in \cite{BorkarGaitsgory}. Note that in the case
considered here,  the fast motion is restricted to remain in a
compact set at all times, the dynamics are affine in the control,
$\sigma$ is uniformly nondegenerate and the functions $c$ and
$\sigma$ do not depend on $x$. For the construction of the control
and precise statements we refer the reader to
\cite{Gaitsgory,BorkarGaitsgory}.
\end{proof}

\begin{remark}
\label{R:MultivaluedMapRegime3_2}

\begin{enumerate}
\item {The difficulties that arise in Regime $3$ are due to the fact that one
has to average with respect to a first order operator. In this case uniqueness
of an invariant measure is not guaranteed and is actually difficult to verify
in practice.}

\item Suppose that {the functions $c$ and $\sigma$ depend on $x$ as well. It
turns out that under some additional Lipschitz type conditions in $x$, one can
still use the methodology in \cite{Gaitsgory,BorkarGaitsgory}. These
conditions are automatically satisfied for any admissible control if the
functions $c(x,y)$ and $\sigma(x,y)$ do not depend on $x$. However, we were
unable to verify them when the coefficients depend on $x$ without imposing any
further restrictions on the class of controls under consideration. For a more
detailed discussion see \cite{Gaitsgory,BorkarGaitsgory}.}
\end{enumerate}
\end{remark}

\subsection{Regime 3: An alternative expression for the rate function in
dimension $d=1$.}

\label{SS:Regime3alternativeFormula}

In this subsection we give an alternative expression of the rate function for
Regime $3$ in dimension $d=1$. The proof is analogous to the proof of the
statement for Regime $2$. We therefore only state the result without proving
it. The reason one can prove the LDP for $d=1$ with the coefficients depending
on $x$ is that the invariant measure takes an explicit form. Then, the local
rate function is the value function to a calculus of variations problem which
can be analyzed by standard techniques. In particular, because everything can
be written explicitly, we can easily prove that the infimum of this
variational problem is attained at a control $\bar{u}$ for which the
corresponding ODE has a unique invariant measure.

Consider a control $v(\cdot):\mathcal{Y}\mapsto\mathbb{R}$. Without loss of
generality one can restrict attention to controls that give nonzero velocity
everywhere. The control $v$ might depend on $(t,x)$ as well, but we omit
writing it for notational convenience. Decomposing the limiting occupation
measure as stochastic kernels (as it was done for Regimes $1$ and $2$) and
fixing the velocity $\beta=\dot{\phi}_{t}$, equations
(\ref{Eq:AccumulationPointsProcessViable}) and
(\ref{Eq:AccumulationPointsMeasureViable}) with $(\lambda,\mathcal{L}%
_{z,x})=(\lambda_{3},\mathcal{L}_{z,x}^{3})$ imply that the corresponding
invariant measure $\mu_{v}(dy)$ that satisfies
(\ref{Eq:AccumulationPointsMeasureViable}) takes the form
\[
\mu_{v}(dy)=\frac{\beta}{c(x,y)+\sigma(x,y)v(y)}dy.
\]
For $x,\beta\in\mathbb{R}$ define
\[
J_{x,\beta}(v)=\frac{1}{2}\int_{\mathbb{T}} |v(y)|^{2}\frac{\beta
}{c(x,y)+\sigma(x,y)v(y)}dy,
\]
and the local rate function
\[
L^{o}_{3}(x,\beta)=\inf_{v}\left\{  J_{x,\beta}(v): \int_{\mathbb{T}}
\frac{\beta}{c(x,y)+\sigma(x,y)v(y)}dy=1\right\}  .
\]

\begin{theorem}
\label{T:MainTheorem4} Assume Condition \ref{A:Assumption1} and that we are
considering Regime $3$. Let $\{X^{\epsilon},\epsilon>0\}$ be the
$1$-dimensional diffusion process that satisfies (\ref{Eq:LDPandA1}). Then
$\{X^{\epsilon},\epsilon>0\}$ satisfies the large deviations principle with
rate function
\[
S(\phi)=%
\begin{cases}
\int_{0}^{1}L_{3}^{o}(\phi_{s},\dot{\phi}_{s})ds & \text{if }\phi
\in\mathcal{C}([0,1];\mathbb{R})\text{ is absolutely continuous}\\
+\infty & \text{otherwise.}%
\end{cases}
\]

\end{theorem}

We conclude this section with the following corollary. As can be easily seen
from the form of $L_{3}^{o}$ in Theorem \ref{T:MainTheorem4}, in the case
$c(x,y)=0$ one obtains a closed form expression for the rate function.

\begin{corollary}
\label{C:MainTheorem5}In addition to the conditions of Theorem
\ref{T:MainTheorem4}, assume that $c(x,y)=0$. Then $\{X^{\epsilon}%
,\epsilon>0\} $ satisfies the large deviations principle with rate function
\[
S(\phi)=%
\begin{cases}
\frac{1}{2}\int_{0}^{1}|\dot{\phi}_{s}|^{2}\left\vert \int_{\mathbb{T}}%
(\sigma^{2}(\phi_{s},y))^{-1/2}dy\right\vert ^{2}ds & \text{if }\phi
\in\mathcal{C}([0,1];\mathbb{R})\text{ is absolutely continuous},\\
+\infty & \text{otherwise.}%
\end{cases}
\]

\end{corollary}

\section{Acknowledgements}
We would like to thank Hui Wang for his initial involvement in this project.

\end{document}